\documentclass[12pt]{amsart}
\usepackage{amsfonts,amsmath,amscd}
\usepackage{epsfig,graphics}
\usepackage{amssymb}
\usepackage{amscd}
\usepackage{hyperref}
\usepackage{bbm}
\usepackage{mathabx}

\newcommand{\BN}{{\mathbb{N}}}
\newcommand{\BR}{{\mathbb{R}}}
\newcommand{\BC}{{\mathbb{C}}}

\newcommand{\gd}{\delta}

\newcommand{\gs}{\sigma} 

\newcommand{\go}{\omega}

\newcommand{\gl}{\lambda}
\newcommand{\ga}{\alpha}

\newcommand{\gTh}{\Theta}

\newcommand{\cC}{{\mathcal{C}}}
\newcommand{\dd}{{\partial}}
\renewcommand{\hat}{\widehat}

\def\k{\kappa}

\def\vp{\varphi}
\def\e{\varepsilon}

\newcommand{\ol}[1]{\overline{#1}}

\newcommand{\Ker}{\text{Ker}}

\newcommand{\tprob}{\text{Prob}}

\newcommand{\half}{\frac{1}{2}}

\newtheorem{prop}{Proposition}[section]
\newtheorem*{prop*}{Proposition}

\newtheorem*{thm*}{Theorem}
\newtheorem{lem}[prop]{Lemma}
\newtheorem{cor}[prop]{Corollary}
\newtheorem*{cor*}{Corollary}

\theoremstyle{definition}

\newtheorem{defn}[prop]{Definition}
\newtheorem*{defn*}{Definition}
\newtheorem{rem}[prop]{Remark}
\newtheorem*{rem*}{Remark}
\newtheorem{exam}[prop]{Example}

\title{Analogs of complementary series for CAT(-1) groups.}

\begin{document}
\maketitle
\centerline{\scriptsize KEVIN BOUCHER}
\centerline{\tiny Weizmann institute of science, 234 Herzl street, Rehovot 7610001, Israel}
\centerline{\tiny (e-mail: kevin.boucher@weizmann.ac.il)}

\begin{abstract}
In this paper we extend the construction of complementary series representations to convex-cocompact isometry groups of CAT(-1) spaces with conditionally negative metrics.
Our approach is purely dynamical and generalizes the constructions, known for negatively curved algebraic groups as $SO(n,1)$, $SU(n,1)$ or $SL_2(\mathbb{Q}_p)$ and their lattices \cite{MR263985}, to new examples as non-linear groups coming from lattices of certain hyperbolic buildings \cite{MR1445387} \cite{MR1668359}.
\end{abstract}

{\tiny \textbf{Key words:} Boundary representations, complementary series, harmonic analysis, Patterson-Sullivan theory.\\
2020 Mathematics Subject Classification: 37A46 20F67 22D10}

\section{Introduction}
Concerning the representation theory and the harmonic analysis over discrete groups arising in a geometric context it is appropriate, at least at one stage, to try to see the problems in a general setting.

Let us therefore begin by recalling the general questions which have guided the development in the past and will certainly continue to serve in this role in the  future:

\begin{enumerate}
\item Given the \textit{natural} geometric action of $G$, is there a procedure that induces irreducible representations over $G$?
\item Is these representations are enough to obtain a type of Plancherel formula for $G$? 
\item What those representations reveal about the harmonic analysis over $G$?
\end{enumerate}

In the case of free groups those questions were investigated for the natural actions over trees by Figà-Talamanca and Piccardelo \cite{MR710827}.
Later Bader Muchnik \cite{MR2787597} pushed the investigation over the broader class of CAT(-1) groups and consider Koopman representations associated to natural dynamical systems over geometric boundaries at infinity called boundary representations.
This new perspectives have motivated many works \cite{Garncarek:2014aa} \cite{Boucher:2020aa} \cite{MR3656473} \cite{MR3549632}.\\

Considering CAT(-1) groups as generalized rank 1 lattices such representations correspond to parabolic inductions from trivial characters.
In this paper we pursue these investigations and straighten this analogy by establishing the existence of complementary series.\\
Given a group $G$, a family of unitary $G$-representations $(\pi_s)_s$ continuously parametrized by the interval $[0,1]$, with respect to the Fell topology on the unitary dual of $G$, such that $\pi_s^*=\pi_{1-s}$ for $s\in[0,1]$, $\pi_\half$ is a boundary representation and $\pi_1={\bf1}_G$ is the trivial representation is called \textit{complementary series} if it satisfies the extra conditions:
\begin{enumerate}
\item For all $s\in[0,1]$, $\pi_s$ is irreducible;
\item for all $t,s\in[0,1]$, $\pi_s$ is unitarily equivalent to $\pi_t$ if and only if $s=t$ or $s=1-t$;
\item if $\gl$ stands for the regular representation of $G$, $\pi_s$ is weakly contained in $\gl$ if and only if $s=\half$.
\end{enumerate}

In the context of rank 1 Lie groups these representations appear as parabolic induction from non-unitary characters over a Cartan subgroup.
As exposed in Section \ref{sec:banachcomp} our approach follows an geometric analogue of this idea.
It leads to a natural family of Banach representations called (s)-homogeneous representations, $(\mathcal{F}_s)_s$, over the skew-product associated to the boundary action (cf. Section \ref{sec:banachcomp}) that are unitarizable under certain assumptions on the geometry of $(X,d)$ and the parameter $s$.

Many negatively curved groups have the Kazhdan property (T) and therefore have their trivial representation isolated \cite{MR2415834}.
In particular the existence of complementary series over the all interval $[0,1]$ rules out those groups although it is conjectured that an uniformly bounded analogue of those complementary series exists on any negatively curved group.\\

To formulate our results we now set up some notations.\\
Let $(X,d)$ be a proper CAT(-1) space and $G$ a non-elementary discrete group of isometries acting properly that is convex-compact.

A continuous real kernel $k$ over $X$ is called conditionally negative \cite{MR2415834} if $k$ is symmetric, satisfies $k(x,x)=0$ for all $x\in X$ and given any finite set $F\subset X$ and any family of real numbers $(c_x)_{x\in F}$ with $\sum_Fc_x=0$ one has:
$$\sum_{x,y\in F}c_xc_yk(x,y)\le0$$ 
Examples of such kernels are given by combinatorial distances over the 1-skeleton of CAT(0) cube complexes, hyperbolic distances over real or complex hyperbolic spaces or more generally pseudo-distances associated to measured wall structures \cite{MR2106770}.
Let us also mention another example given by the distances of  \textit{even} Bourdon hyperbolic building $\text{I}_{4p',q}$ with $p'\ge2$ and $q\ge 3$ \cite{MR1445387} \cite{MR1668359}.\\

{\bf Theorem 1.}
\textit{Assume the distance $d$ over $X$ is conditionally negative, then the family of homogeneous Banach representations $(\mathcal{F}_s)$ are unitarizable over the all interval $[0,1]$.\\
If $\mathcal{H}_s$ stands for the unitarization of $\mathcal{F}_s$ one has in addition $\mathcal{H}_s^*=\mathcal{H}_{1-s}$ for $s\in[0,1]$, $\mathcal{H}_\half$ is a boundary representation and $\mathcal{H}_1$ is the trivial representation.}\\

Let $\mu_o$ be the Patterson-Sullivan measure at a basepoint $o\in X$ over the boundary $\dd X$ of $X$.
The proof of theorem 1 relies on the spectral positivity of a kernel operator $\mathcal{I}_s$ with $\half<s\le1$ defined over $L^2(\dd X,\mu_o)$ introduced in Subsection \ref{sub:int} and inspired by \cite{MR710827}.
More precisely the conditional negativity of the metric, $d$, over $X$ implies the positivity of $\mathcal{I}_s$.\\

A remarkable fact is that the positivity of $\mathcal{I}_s$ is the only obstruction for $\mathcal{F}_s$ to form a complementary series over $G$:\\

{\bf Theorem 2.}
\textit{Assume the operators $\mathcal{I}_s$ are positive over $L^2(\dd X,\mu_o)$ for all $s\in [\half,1]$, then the unitary $G$-representations $(\mathcal{H}_s)_{s\in[0,1]}$ form a complementary series.}\\

Note that even if the positivity of the operator $\mathcal{I}_s$ is only established for conditionally negative distances we conjecture that it remain true whenever the trivial representation of $G$ is not isolated.


\subsection{Outlines}$ $\\
After introducing standard facts about hyperbolic geometry and the Patterson-Sullivan theory in Section \ref{sec:prelim}, we introduced the family of Banach representations $\mathcal{F}_s$ and the operators $\mathcal{I}_s$, discussed above, in Section \ref{sec:banachcomp}.
The Section \ref{sec:pos} is dedicated to the proof of Theorem 1.
In Section \ref{sec:anamat} we analyze matrix coefficients of certain averages of operators related to those representations. 
This is our principal tool to investigate the properties of the representations $(\mathcal{H}_s)_{s\in[0,1]}$ in Section \ref{sec:end} where we prove Theorem 2.

\subsection{Acknowledgments}$ $\\
We wish to thank Uri Bader, Hengfei Lu and Andrzej Zuk for their remarks and helpful discussions.
We are also grateful to Adrien Boyer for introducing us to the broad subject of boundary representations.

\subsection{Notations and conventions}$ $\\
In order to avoid the escalation of constants coming from estimates up to controlled additive or multiplication error terms we will use the following conventions.
Given two real valued functions, $a,b$, over a set $Z$, we write $a\preceq b$ if there exists $C>0$ such that $a(z)\le Cb(z)$ for all $z\in Z$ and $a\asymp b$ if $a\preceq b$ and $b\preceq a$.
Analogously we write $a\lesssim b$ if there exists $c$ such that $a(z)\le b(z)+c$ and $a\simeq b$ if $a\lesssim b$ and $b\lesssim a$.\\

Given a topological space $Z$, the space of continuous functions over $Z$ is endowed with topology of uniform convergence on compact sets if nothing else is specified and is denoted $\cC(Z)$.
The space of Borel probability over $Z$ is considered with its $*$-weak topology. 
If $Z$ has a metric structure, $d$, we denote $\text{dim}_H(Z)$ its Hausdorff dimension of $(Z,d)$.\\

The notations $r$ and $R$ stands for specific constants introduced in subsection \ref{sub:cov} and $o\in X$ for a fixed basepoint.

\section{preliminaries}\label{sec:prelim}
In this section we introduce some material and facts needed throughout this paper.
The reader can refer to \cite{MR1214072} \cite{MR1325797} \cite{MR1341941} for further details.\\

A proper geodesic space $(X,d)$ is CAT(-1) if for any geodesic triangle $\Delta$ its comparison $\ol{\Delta}$ in $\mathbb{H}^2$ satisfies:
$$d(x,y)\le d_{\mathbb{H}^2}(\ol{x},\ol{y})$$
for all $x,y\in \Delta$.
This notion encapsulates a large family of spaces as Riemannian manifolds with sectional curvature bounded above by $-1$ as well as hyperbolic buildings \cite{MR1445387}.

\subsection{Hyperbolic spaces and compactifications.}
$ $\\

Given a basepoint $o\in X$, the \textit{Gromov product} over $X$ at $o\in X$ is defined as:
$$\langle x,y\rangle_o=\half(d(x,o)+d(y,o)-d(x,y))$$
for $x,y\in X$.
A space $(Z,d)$ is called Gromov hyperbolic if one can find a positive constant $c_Z\ge0$ such that:
$$\langle x,y\rangle_w\ge\min\{\langle x,z\rangle_w,\langle y,z\rangle_w\}-c_Z$$
for all $x,y,z,w\in Z$.
Every proper geodesic CAT(-1) space satisfies this condition \cite{MR1744486}.\\

Given a hyperbolic space $(X,d)$ a boundary, $\dd X$, called Gromov boundary is attached to it.
Let us recall a construction of this object and its properties.

\subsubsection{Hyperbolic boundaries regarded as equivalence classes of sequences}
$ $\\
A sequence of points $(x_n)_n$ in $X$ \textit{converges to infinity} if $\lim_{n,m}\langle x_n,x_m\rangle_o\rightarrow+\infty$.\\
Since $(X,d)$ is hyperbolic the relation defined over the set of sequences which converge to infinity, $X^\infty\subset X^\BN$ as:
$$(a_n)\mathcal{R}(b_n)\quad \text{if and only if}\quad \lim_{n,m}\langle a_n,b_m\rangle_o\rightarrow\infty$$
with $(a_n)_n,(b_n)_n\in X^\infty$, is a equivalence relation independent of the basepoint $o$.
The boundary $\dd X$ of $X$ is defined as the set of equivalence classes for this relation and the class of a sequence $(x_n)_n\in X^\infty$ is denoted $\lim_nx_n=\xi$.\\

The Gromov product extends to $\ol{X}=X\cup\dd X$ and in the case of CAT(-1) spaces this extension is continuous and given by the formula:
$$\langle x,y\rangle_o=\lim_{n,m}\langle x_n,y_n\rangle_o\in[0,+\infty]$$
where the sequences $(x_n)_n$ and $(y_m)_m$ converge respectively to $x$ and $y$ in $\ol{X}=X\cup\dd X$.

This extended Gromov product satisfies \cite{MR1744486}:
\begin{enumerate}
\item $\langle x,y\rangle_o=\infty$ if and only if $x,y\in\dd X$ and $x=y$;
\item $\langle x,y\rangle_o\ge\min\{\langle x,z\rangle_o,\langle z,y\rangle_o\}-2c_X$ for all $x,y,z\in \ol{X}$.\\
\end{enumerate}

Given $x\in X$ the kernel:
$$d_x:(\xi,\eta)\in\dd X\times\dd X\mapsto 
\begin{cases}
e^{-\langle\xi,\eta\rangle_x}\,\text{when $\xi\neq\eta$}\\
0\,\text{otherwise}
\end{cases}
$$
defines a distance over $\dd X$ called visual distance at $x$ \cite{MR1445387} that defines the natural topology over $\dd X$ independent of $x\in X$.

The space $\ol{X}=X\cup\dd X$ admits a compact metrizable topology that is compatible with the topology of $X$ and $\dd X$ such that 
a sequence $(x_n)_n$ in $\ol{X}$ converges to $\xi\in\dd X$ if and only if $\langle x_n,\xi\rangle_o\rightarrow\infty$.
In particular given a function $\vp\in\cC(\ol{X})$ its uniform continuity around points at infinity can be expressed as:
$$|\vp(\xi)-\vp(x)|\prec \go(\langle \xi,x\rangle_o)$$
for all $x\in\ol{X}$ and $\xi\in\dd X$, where $\go$ is a positive decreasing function with $\go(t)\xrightarrow{t\rightarrow+\infty}0$.\\

A essential object in our framework is the Busemann function $b$.
Given $\xi\in\dd X$, any geodesic $(\xi_t)_t$ that converges to $\xi$ and $x,y\in X$, $b_\xi(x,y)$ is defined as $\lim_td(x,\xi_t)-d(y,\xi_t)$.
The map $b_\bullet(\bullet,\bullet):\ol{X}\times X\times X\rightarrow\BR$ is continuous and satisfies:
$$\langle x,y\rangle_o=\half\sup_z[b_x(o,z)+b_y(o,z)]=\half[b_x(o,p)+b_y(o,p)]\in\BR_+$$
where $p$ belongs to the geodesic between $x,y\in \ol{X}$. 
Moreover for any $x\in \ol{X}$ the map $b_x:X\times X\rightarrow \BR$ is a additive cocycle over $X$.
As a consequence the family of visual distances $\{d_x\}_{x\in X}$ satisfies the conformal relation:
$$d_x(\xi,\eta)=e^{\half[b_\xi(x,x')+b_\eta(x,x')]}d_{x'}(\xi,\eta)$$
 for any $\xi,\eta\in \dd X$ and $x,x'\in X$.\\

The group of isometries, $\text{Is}(X)$, of $(X,d)$ is endowed with the compact-open topology that makes it locally compact and second countable.
Under these assumptions the action of $\text{Is}(X)$ on $(X,d)$ is continuous, has closed orbits and compact stabilizers.
Its action extends continuously to $\ol{X}$ by homeomorphisms and
satisfies the invariance relations: 
$$b_{g.x}(g.y,g.z)=b_x(y,z)$$ 
and
$$d_{g.x}(g.\xi,g.\eta)=d_{x}(\xi,\eta)$$
for any $g\in\text{Is}(X)$, $x\in\ol{X}$, $y,z\in X$ and $\xi,\eta\in \dd X$.\\




\subsection{Boundary retractions}
$ $\\
The limit set of a discrete subgroup $G\subset \text{Is}(X)$ is defined as $\Lambda(G)=\ol{G.o}\cap\dd X$.
Given $\Lambda(G)$ the geodesic hull $\mathcal{Q}(\Lambda(G))\subset X$ is the union of geodesics with endpoints in $\Lambda(G)$.
The group $G$ is called convex-cocompact if it acts cocompactly over $\mathcal{Q}(\Lambda(G))$ or equivalently if there exists a uniform constant $R_{X,G,o}$ such that any geodesic from $o$ to $\xi\in\Lambda(G)$ stays within a $R_{X,G,o}$-neighborhood of $G.o$. In particular cocompact subgroups are convex-cocompact \cite{MR1214072} \cite{Sullivan_1979}.\\ 

A measurable map $f:\ol{\mathcal{Q}(\Lambda(G))}\to\Lambda(G)$ is called \textit{boundary retraction} if
one can find $c_{X,G,f,o}$ such that $\langle f(x),x\rangle_o\ge d(o,x)-c_{X,G,f,o}$ and $f|_{\Lambda(G)}={\bf I}_{\Lambda(G)}$.
In particular $f$ is continuous at any point of $\Lambda(G)$.

The \textit{shadow} at $x\in X$ of radius $r$ from the basepoint $o\in X$ is defined as:
$$\mathcal{O}_o(x,\rho)=\{\xi\in\dd X|\,\text{$\langle\xi,x\rangle_o\ge d(x,o)-\rho$}\}$$
There exists $r_{X,o}>0$ large enough such that $\mathcal{O}_o(x,r_{X,o})\neq\emptyset$ for all $x\in G.o$. 
Using a compact fundamental domain for the action of $G$ over $\mathcal{Q}(\Lambda(G))$ one can show that a boundary retraction exists  (cf \cite{Boucher:2020aa}, \cite{MR3939562}).


For the rest of this paper a boundary retraction, $f$, is fixed.
Given $x\in \mathcal{Q}(\Lambda(G))$ and $g\in \text{Isom}(X,d)$ we denote $\hat{x}=f(x)\in\dd X$, $\hat{g}=f(g.o)\in\dd X$ and $\widecheck{g}=f(g^{-1}.o)\in\dd X$.



\subsection{The Patterson-Sullivan theory}$ $\\ \label{sec:disin}
Let $G\subset \text{Is}(X)$ be a discrete convex-cocompact subgroup of isometries of $(X,d)$ a proper CAT(-1) space that is non-elementary, i.e. $G$ is not virtually abelian.

\begin{defn}
The critical exponent of $G$ denoted $\gd\in[0,+\infty]$, is defined as the infimum over all $s\ge0$ such that the Poincaré integral at $s$:
$$\int_Ge^{-sd(g.o,o)}dg$$
is finite.
\end{defn}

It follows from \cite{MR1214072} that $G$ has finite critical exponent and is divergent, i.e. its Poincaré integral diverges at $s=\gd$.

\begin{defn}
A $\ga$-density, for some positive $\ga$, is a continuous $G$-equivariant map:
$$\mu_\bullet:X\mapsto \tprob(\dd X);\quad x\mapsto \mu_x$$
such that for all $x,y\in X$, $\mu_x\sim\mu_y$ and
$$\frac{d\mu_x}{d\mu_y}(\xi)=e^{-\ga b_\xi(x,y)}$$
for $[\mu]$-almost every $\xi\in\dd X$.
\end{defn}

Let us denote $m$ the push-forward of the Haar measure of $G$ over the orbit $G.o$.
A $\gd$-density over $\dd X$ can be constructed from $m$ as follows.\\
Let
$$\mu_{\bullet,t}:X\rightarrow\tprob(\ol{X});\quad x\mapsto\mu_{x,t}$$ 
be the map defined as:
$$d\mu_{x,t}(y)=\frac{1}{\mathcal{W}(t)}e^{-td(x,y)}dm(y)$$
where
$$\mathcal{W}(t)=\int_Xe^{-td(x,y)}dm(y)$$
and $t>\gd$.\\
The family $(\mu_{\bullet,t})_{t>\gd}$ is $G$-equivariant, equicontinuous and bounded on compact sets for any $t\in(\gd,\gd+1]$ \cite{MR1325797}. 
It is therefore pre-compact as elements of $\cC_G(X,\tprob(\ol{X}))$ and any of its accumulation point defines a $\gd$-density supported on the limit set of $G$, $\ol{G.o}\cap\dd X=\Lambda(G)$.

Under the assumptions over $G$ $\mu_\bullet$ is unique, non-atomic and ergodic \cite{MR1214072}.\\

We conclude this subsection with a alternative form of the Shadow lemma for the family $(\mu_{\bullet,t})_{t>\gd}$.
We shall need the following lemma.
\begin{lem}\label{lem:preshadow}
Given $0<\e\le1$, there exist $\gd<T$ and $r_1=r_1(o)$ such that:
$$\mu_{o,t}(\mathcal{C}(o,x,\rho))\le \e$$
where
$$\mathcal{C}(o,x,\rho)=\{z\in \ol{X}\,|\,\langle y,x\rangle_o\ge \rho\}\subset \ol{X}$$
for all $\gd<t\le T$, $x\in\ol{X}$ and $\rho\ge r_1$.
\end{lem}
\begin{proof}
Assume one can find $0<\e_0\le1$, $(\mathcal{C}_n=\mathcal{C}(o,x_n,\rho_n))_n$ with $\rho_n\rightarrow+\infty$ and $(t_n)_n$ with $t_n>\gd$ and $t_n\rightarrow \gd$ such that:
$$\mu_{o,t_n}(\mathcal{C}(o,x_n,\rho_n))\ge \e_0.$$
Then, up to a subsequence, we can assume that the sequence of closed sets, $(\mathcal{C}_n)_n$, inside of the metrizable compact space $\ol{X}$ converges for the Gromov-Hausdorff topology to $\mathcal{C}_\infty$.
Since $\rho_n\le\langle z,x_n\rangle_o\le d(o,z)$ for all $z\in\mathcal{C}_n\cap X$ one has $\mathcal{C}_\infty\subset\dd X$.
Moreover for all $\xi,\eta\in\mathcal{C}_\infty$ one can find $z_n\in\mathcal{C}_n\rightarrow\xi$ and $z'_n\in\mathcal{C}_n\rightarrow\eta$ it follows:
$$\langle\xi,\eta\rangle_x=\lim_n\langle z_n,z_n'\rangle_x\gtrsim\lim_n\min\{\langle z_n,x_n\rangle_o,\langle x_n,z_n'\rangle_o\}\ge \lim_n\rho_n=+\infty$$
in other words $\mathcal{C}_\infty$ is reduced to a single point.
Using the fact that $\mathcal{C}_n$ is Cauchy for the Gromov-Hausdorff distance for any $i$ one can find $N(i)$ such that:
$$\mathcal{C}_{n+k}\subset\mathcal{V}(\mathcal{C}_n,2^{-i})$$
for all $n\ge N(i)$ which implies for any $k\ge0$:
$$\mu_{o,t_{N(i)+k}}(\mathcal{V}(\mathcal{C}_{N(i)},2^{-i}))\ge \mu_{o,t_{N(i)+k}}(\mathcal{C}_{N(i)+k})\ge\e_0$$
At the limit over $k$ it follows:
$$\mu_{o}(\mathcal{V}(\mathcal{C}_{N(i)},2^{-i}))\ge\e_0$$
for all $i\ge0$ and therefore  $\lim_i\mu_{o}(\mathcal{V}(\mathcal{C}_{N(i)},2^{-i}))=\mu_{o}(\mathcal{C}_\infty)\ge\e_0$ which contradict the fact that $\mu_o$ is atom free.
\end{proof}

\begin{lem}[Shadow lemma]\label{lem:equi1}$ $\\
There exist $T>\gd$ and $r_0=r_0(o)$ such that:
$$\mu_{o,t}(\mathcal{U}(o,x,\rho))\asymp e^{\gd \rho}e^{-\gd d(o,x)}$$
where 
$$\mathcal{U}(o,x,\rho)=\{z\in \ol{X}\,|\,\langle z,x\rangle_o\ge d(o,x)-\rho\}\subset \ol{X}$$
$o\in X$, $x\in\mathcal{Q}(\Lambda(G))$, $\rho\ge r_0$ and $\gd<t\le T$.\\
In particular:
$$\mu_{o,t}(\mathcal{C}(o,\xi,\rho))\asymp e^{-\gd \rho}$$
for all $\xi\in\Lambda(G)$, $\rho\ge r_0$ and $\gd<t\le T$.
\end{lem}
\begin{proof}
It is enough to prove:
$$\half e^{-\gd d(g.o,o)}\le \mu_{o,t}(\mathcal{U}(o,g.o,c))\le e^{2\gd c}e^{-\gd d(g.o,o)}$$
for any $g\in G$ and $c\ge r_0$ and $\gd<t\le T$.
Indeed using a tree approximation argument (cf. \cite{MR1214072} Proposition 7.4 proof) one can find universal constants $c_1$ and $c_2$ such that for any $x\in X$ and $\rho\ge r_0$:
$$\mathcal{U}(o,g(x).o,c_1)\subset \mathcal{U}(o,x,\rho)\subset\mathcal{U}(o,g(x).o,c_2)$$
where $g(x)\in G$ with $d(o,g(x).o)\simeq d(o,x)-\rho$ if $x\in X$ and $d(o,g(x).o)\simeq \rho$ otherwise.\\

Given $g\in G$ one has:
\begin{align*}
\mu_{o,t}&(\mathcal{U}(o,g.o,\rho))=\mu_{o,t}(g.\mathcal{U}(g^{-1}.o,o,\rho))\\
&=\mu_{g^{-1}.o,t}(\mathcal{U}(g^{-1}o,o,\rho))=\int_{\mathcal{U}(g^{-1}o,o,\rho)}e^{-\gd b_z(g^{-1}.o,o)}d\mu_{o,t}(z)
\end{align*}
Since $b_z(g^{-1}.o,o)=2\langle z,o\rangle_{g^{-1}o}-d(o,g^{-1}.o)$ for all $g\in G$, $o\in X$ and $z\in\ol{X}$ it follows:
$$d(o,g.o)-2\rho\le b_z(g^{-1}.o,o)\le d(o,g.o)$$
for all $z\in \mathcal{U}(g^{-1}.o,o,\rho)$.
Therefore:
$$e^{-\gd d(g.o,o)}\mu_{o,t}(\mathcal{U}(g^{-1}o,o,\rho))\le \mu_{o,t}(\mathcal{U}_{o}(o,g.o,\rho))\le e^{2\gd \rho}e^{-\gd d(g.o,o)}\mu_{o,t}(\mathcal{U}(g^{-1}o,o,\rho))$$

On the other hand, for any $y,z\in \mathcal{U}(g^{-1}o,o,\rho)^c\subset\ol{X}$ one has:
$$\langle y,o\rangle_{g^{-1}.o}=d(g.o,o)-\langle y,g^{-1}.o\rangle_{o}$$
and similarly for $z$.
It follows:
$$\langle y,g^{-1}o\rangle_{o},\langle z,g^{-1}o\rangle_{o}\ge \rho.$$
Using the hyperbolic inequality one obtain:
$$\langle y,z\rangle_o\ge \min\{\langle g^{-1}.o,z\rangle_o,\langle g^{-1}.o,y\rangle_o\}-c_X\ge \rho-c_X$$
for some universal constant $c_X\ge0$.
In other words $\ol{X}\setminus \mathcal{U}(g^{-1}o,o,\rho)\subset \mathcal{C}(o,y,\rho-c_X)$ for any $y\in \mathcal{U}(g^{-1}o,o,\rho)^c$.
Lemma \ref{lem:preshadow} implies that one can find $r_1$ such that:
$$\mu_{o,t}(\ol{X}\setminus \mathcal{U}(g^{-1}o,o,\rho))\le\half $$
for all $\rho\ge r_1$ and thus:
$$\half e^{-\gd d(g.o,o)}\le \mu_{o,t}(\mathcal{U}_{o}(o,g.o,\rho))\le e^{2\gd \rho}e^{-\gd d(g.o,o)}$$
\end{proof}

As a consequence of Lemma \ref{lem:equi1} one has $\mu_o(\mathcal{O}_o(x,\rho))\neq 0$ for all $x\in \mathcal{Q}(\Lambda(G))$ and $\rho>r_0$. 

\subsection{Covering and equidistribution}$ $\\
\label{sub:cov}


The visual ball centered at $\xi\in \dd X$ of radius $e^{-\rho}$ is defined as:
$$B_o(\xi,e^{-\rho})=\{\eta\in\dd X\,|\,\langle \xi,\eta\rangle_o\ge \rho\}=\{\eta\in\dd X\,|\,d_o(\xi,\eta)\le e^{-\rho} \}$$

We introduce the straight and inverted $\varsigma$-cones from $o\in X$ to $\xi\in\dd X$ of radius $\rho$ respectively as:
$${\bf C}^+_{o}(\xi;\rho,\varsigma)=\{g\in G \,|\,\mathcal{O}_o(g.o,\varsigma)\cap B_o(\xi,e^{-\rho})\neq\emptyset\}$$
$${\bf C}^-_{o}(\xi;\rho,\varsigma)=\{g\in G \,|\,\mathcal{O}_o(g^{-1}.o,\varsigma)\cap B_o(\xi,e^{-\rho})\neq\emptyset\}$$
The two sided cone from $o\in X$ at $(\xi,\eta)\in\dd X\times\dd X$ is defined as:
$${\bf C}^{(2)}_{o}(\xi,\eta;\rho,\varsigma)={\bf C}^+_{o}(\xi;\rho,\varsigma)\cap{\bf C}^{-}_{o}(\eta;\rho,\varsigma)$$

Since $G$ is convex-cocompact there exists $R_{X,G,o}>0$ such that every geodesic from a basepoint $o\in X$ to $\xi\in \Lambda(G)$ is at distance at most $R_{X,G,o}$ from $G.o$. 
Let us fix once for all a universal constant $R\ge R_{X,G,o}$ such that for all $g,h\in G$, there exists $g'\in G$ with $d(g.o,g'.o)\le R$ and $d(gh.o,o)\ge d(g.o,o)+d(h.o,o)-2R$.
The existence of such a constant follow from Milnor lemma together with \cite{Garncarek:2014aa} Lemma 4.4.
In the rest $R$ will refer to this particular constant and we define:
$$S_{o,G}(t)=\{g\in G\,|\,tR\le d(g.o,o)<(t+1)R\}\subset G.$$


The following counting lemma corresponds to Lemma 4.3 of \cite{Garncarek:2014aa}:
\begin{lem}\label{lem:count}
Let $r_0$ be as in Lemma \ref{lem:equi1}. 
There exists $r\ge r_0$ such that for all $\xi\in\Lambda(G)$ and $t\ge \frac{\rho}{R}$:
 $$|{\bf C}^+_{o}(\xi;\rho,r)\cap S_{o,G}(t)|\asymp_{o,r} e^{\gd tR}e^{-\gd \rho}$$
\end{lem}
In the rest $r$ will refer to this particular constant.\\

The following covering and equidistribution results are slight improvements of results introduced in \cite{MR2787597}.

\begin{lem}[Vitali cover lemma]\label{lem:vitali}
There exists $r'>r$ such that for all $t>0$ there exist $S^*_{o,G}(t)\subset S_{o,G}(t)$ and a family of  measurable subsets $(O^{(2)}_{o}(g))_{g\in S^*_{o,G}(t)}$ of $\dd X\times\dd X$ such that:
\begin{enumerate}
\item $|S^*_{o,G}(t)|\asymp e^{\gd tR}$;
\item $\bigcup_{g\in S^*_{o,G}(t)}O^{(2)}_{o}(g)\supseteq\Lambda(G)\times\Lambda(G)$;
\item $O^{(2)}_{o}(g)\cap O^{(2)}_{o}(h)=\emptyset$ for all $g\neq h$ in $S^*_{o,G}(t)$;
\item $$\mathcal{O}_o(g.o,R(t))\times \mathcal{O}_o(g^{-1}.o,R(t))\subset O^{(2)}_{o}(g)$$
and
$$O^{(2)}_{o}(g)\subset \mathcal{O}_o(g.o,R'(t))\times \mathcal{O}_o(g^{-1}.o,R(t)')$$
where $R(t)=\half tR+r$ and $R'(t)=\half tR+r'$.
\end{enumerate}
\end{lem}
\begin{proof}
Using Lemma 4.2 \cite{Garncarek:2014aa} one has:
$$\Lambda(G)\times\Lambda(G)\subset\bigcup_{g\in S_{o,G}( t)} \mathcal{O}_o(g.o,R(t))\times \mathcal{O}_o(g^{-1}.o,R(t))$$
Since $\mathcal{O}_o(g.o,r)\subset B_o(\hat{g},e^{-\gd(R(t)-C)})$ for some $C\ge0$ and $g\in S_{o,G}( t)$ one has:
$$\Lambda(G)\times\Lambda(G)\subset\bigcup_{g\in S_{o,G}(t)} B_o(\hat{g},e^{-\gd(R(t)-C)})\times B_o(\widecheck{g},e^{-\gd(R(t)-C)})$$
Using Vitali Lemma \cite{MR1232192} one can find $S^*_{o,G}(t)\subset S_{o,G}(t)$ such that:
$$\Lambda(G)\times\Lambda(G)\subset\bigcup_{g\in S^*_{o,G}(t)} B_o(\hat{g},5e^{-\gd(R(t)-C)})\times B_o(\widecheck{g},5e^{-\gd(R(t)-C)})$$
and
$$B_o(\hat{g},e^{-\gd(R(t)-C)})\times B_o(\widecheck{g},e^{-\gd(R(t)-C)})\bigcap B_o(\hat{h},e^{-\gd(R(t)-C)})\times B_o(\widecheck{h},e^{-\gd(R(t)-C)})=\emptyset$$
for all $g\neq h$ in $S^*_{o,G}(t)$.
Following \cite{MR1232192} Lemma 2 p.15 we define $O^{(2)}_{o}(g)$ that satisfies $(2)$, $(3)$ and $(4)$ with $g\in S^*_{o,G}(t)$ by induction as:
$$O^{(2)}_{o}(g_k)=B^{(2)}_o(g,5e^{-\gd(R(t)-C)})\bigcap[\bigcup_{j<k}O^{(2)}_{o}(g_j)]^c\bigcap[\bigcup_{k<j}B^{(2)}_o(g_j,e^{-\gd(R(t)-C)})]^c$$
where we enumerate $S^*_{o,G}(t)=\{g_1,\dots g_{|S^*_{o,G}(t)|}\}$ and $B^{(2)}_o(g,5e^{-\gd(R(t)-C)})$ stands for $B_o(\hat{g},e^{-\gd(R(t)-C)})\times B_o(\widecheck{g},e^{-\gd(R(t)-C)})$.
It follows from $(4)$ that
$$1=\mu_o\otimes \mu_o(\Lambda(G)\times\Lambda(G))=\sum_{g\in S^*_{o,G}(t)}\mu_o\otimes \mu_o(O^{(2)}_{o}(g))\asymp |S^*_{o,G}(t)|e^{-\gd tR}$$
and thus
$$|S^*_{o,G}(t)|\asymp e^{\gd tR}$$
\end{proof}

As a consequence of the Lebesgue differentiation theorem for Ahlfors regular measures:
\begin{cor}
The sequence  of probabilities $(\nu_{o,t})_{t>0}$ over $G$ defined as:
$$\nu_{o,t}=\int_{S^*_{o,G}(t)}\mu_o(O^{(2)}_{o}(g)).\mathcal{D}_{g}dg$$
,where $\mathcal{D}_{g}$ stands for the Dirac mass at $g\in G$, are supported over $S^*_{o,G}(t)$ for any $t>0$ and satisfies:
$$\int_G\Psi(g.o,g^{-1}.o)d\nu_{o,t}(g)\rightarrow \int_{\dd X\times\dd X}\Psi(\xi,\eta)d\mu_o(\xi)d\mu_o(\eta)$$
for any $\Phi\in\cC(\dd X\times \dd X)$.
\end{cor}
The proof follows the arguments of Theorem 3.2 of \cite{MR3939562}.
\section{Parabolic induction from a geometric perspective}\label{sec:banachcomp}
As above $(X,d)$ stands for a proper CAT(-1) space, $G$ a discrete non-elementary convex-cocompact subgroup of isometries and $\mu_\bullet$ its $\gd$-conformal density over $\dd X$.

Under our assumptions the probability $\mu_o$ over $\dd X$ has Hausdorff dimension $\gd$ which means that:
$$\sup_\xi\int_{\dd X} \frac{1}{d_{o}^{\ga}(\xi,\eta)}d\mu_o(\eta)<+\infty$$
 for all $0\le\ga<\gd$ \cite{MR1878587} \cite{Sullivan_1979}.

\subsection{Skew product and Homogeneous functions}$ $\\
The skew-product $\mathcal{H}(X)=\dd X\times\BR_+^*$ is the $G$-dynamical system defined as 
$g.(\xi,t)=(g.\xi,\frac{dg^{-1}\mu_o}{d\mu_o}(\xi)t)$ with $g\in G$ and $(\xi,t)\in\mathcal{H}(X)$ together with the infinite $G$-invariant measure $d\hat{\mu}(\xi,t)=d\mu_o(\xi) \frac{dt}{t^2}$.\\

\textit{The guideline of our approach is the parabolic induction:\\
As an example let $G=SO(n,1)$, $n\ge1$, be the isometry group of the symmetric space $X=\mathbb{H}^n=G/K$ where $K\simeq SO(n)$ is a maximal compact subgroup of $G$ and $\dd X\simeq G/P$ the Furstenberg boundary of $\mathbb{H}^n$ with $P$ a minimal parabolic subgroup.\\
Denote $A\simeq(\BR^*_+,\times)$ a Cartan subgroup of $G$, $N$ a maximal unipotent subgroup and $M= K\cap P$.
A unitary character, $\chi$, over $A$ and a irreducible representation, $\pi_M$, of $M$ naturally produce a representation of the Levy group $AM\simeq A\times M$ which extends to $P=MAN$. Note that $N$ is normal inside of $P$ and $\chi\times\pi_M$ is taken to be trivial over $N$.\\
The parabolic induction with parameter $(\chi,\pi_M)$ is therefore obtained as a Hilbert completion of the space of sections:
$$\mathcal{F}_{(\chi,\pi_M)}=\{f\in \cC(G;{\bf H}_{\chi\times\pi_M})\,|\,\text{$f(pg)=\chi\times\pi_M(p)f(g)$ and $\ol{\text{supp}(f)P}\subset G/P$ compact}\}$$
where ${\bf H}_{\chi\times\pi_M}$ stands for the representation of $P$.
It is known that in the spherical case, i.e. when $\pi_M=1$, the space of sections $\mathcal{F}_{\chi}$ associated to certain non unitary characters of $A$ admits a $G$-unitary structure.
These representations called unitary spherical complementary series are our objects of investigation.\\}

\textit{
In the spherical case the space of sections identifies with functions, $f$, over $MN\backslash G$ that satisfy $f(aMNg)=\chi(a)f(MNg)$ for all $g\in G$ and $a\in A$.
From a geometric stand point where $X=\mathbb{H}^n$ is regarded as a CAT(-1) space and $\dd X\simeq \mathbb{S}^{n-1}$ as its geometric boundary, the subgroup $P$ is identified with the stabilizer of a point $\xi\in\mathbb{S}^{n-1}$ at infinity, $MN\subset P$ as the stabilizer of the horospheres centered around this point and where $A\subset P$ acts by homogeneous dilations on these horospheres.\\
Our approach consists to identify $({MN}\backslash G,d\ol{g})$ with the skew product $(\mathcal{H}(X),\hat{\mu})$ which corresponds in the geometric setting to the horospherical foliation over the space $X$ \cite{MR2057305}.\\}

A continuous function $f$ on $\mathcal{H}(X)$ is called \textit{$s$-homogeneous} for a complex parameter $s\in \BC$ if
$$f(\xi,tt')=t^sf(\xi,t')$$
for $\hat{\mu}$-almost every $(\xi,t'),(\xi,tt')\in \mathcal{H}(X)$.\\
The space of $s$-homogeneous functions, denoted $\mathcal{F}_s$, is an analogy with the space of $A$-equivariant sections over $MN\backslash G$ for the character $\chi(a)=a^{s}$.\\

These spaces are $G$-invariant and together with the norm:\\
$$\|f\|_{\infty,s}=\text{ess sup}_{(\xi,t)}|t^{-s}f(\xi,t)|$$
for $f\in\mathcal{F}_s$ they define a family of Banach $G$-representations.

Given the character $\chi_s: t\mapsto t^s$ over $(\BR_+^*,\times)$, the operator:
$$e_s:\cC(\dd X)\rightarrow (\cC(\mathcal{H}(X)),\|\,.\,\|_{\infty,s});\quad \vp\mapsto \chi_s\times \vp$$
induces an isometry between $\cC(\dd X)$ and the space of $s$-homogeneous functions with inverse:
$$p:\cC(\mathcal{H}(X))\rightarrow \cC(\dd X);\quad f\mapsto [\xi\mapsto \frac{1}{2\pi}\int_{\BR_+^*} t^{-s}f(\xi,t)\frac{dt}{1+t^2}]$$

Observe that:
$$g.f({\xi},t)=f(g^{-1}.\xi,\frac{dg\mu_o}{d\mu_o}(\xi)t)=[\frac{dg\mu_o}{d\mu_o}(\xi)]^{s}f(g^{-1}.{\xi},t)$$
for all $f\in \mathcal{F}_s$, $g\in G$ and  $\hat{\mu}$-almost every $(\xi,t)\in \mathcal{H}(X)$.

If $\pi_s:G\times \cC(\dd X)\rightarrow \cC(\dd X)$ denote the Banach $G$-representation defined as:
$$\pi_s(g)(\vp)({\xi})=[\frac{dg\mu_o}{d\mu_o}(\xi)]^{s}\vp(g^{-1}.{\xi})$$
for $g\in G$, $\vp\in\cC(\dd X)$ and $\xi\in \dd X$.
The operator $e_s$ defines an isomorphism between the Banach representations $\mathcal{F}_s$ and $(\cC(\dd X),\pi_s)$ that can be seen as a compact realization of the representations $\mathcal{F}_s$ for $s\in\BC$.\\

A natural coupling between $\mathcal{F}_s$ and $\mathcal{F}_{1-s}$ exists :
\begin{lem}\label{lem:dual}
The bilinear map:
$$Q:\mathcal{F}_s\times\mathcal{F}_{1-s}\rightarrow\BC;\quad (f_1,f_2)\mapsto \int_{\dd X} f_1(\xi)\ol{f_2}(\xi)d\mu_o(\xi)$$
is continuous and $G$-invariant.
Moreover the operator $Q:\mathcal{F}_s\rightarrow\mathcal{F}_{1-s}^*$ defines an $G$-intertwiner with 
$$\Ker[Q]=\{f\in\mathcal{F}_s\,|\,f|_{\Lambda(G)}=0\}.$$
\end{lem}
\begin{proof}
Since $\mu_o$ is a probability the continuity of the bilinear map $Q$ is trivial.

For all $s\in\BC$, $\vp\in\mathcal{F}_s$, $\phi\in\mathcal{F}_{1-s}$ and $g\in G$ one has:
\begin{align*}
Q(g.\vp,\phi)&=\int_{\dd X} [\frac{dg\mu_o}{d\mu_o}(\xi)]^{s}\vp(g^{-1}\xi)\ol{\phi}(\xi)d\mu_o(\xi)\\
&=\int_{\dd X} [\frac{dg\mu_o}{d\mu_o}(g.\xi)]^{s}\vp(\xi)\ol{\phi}(g.\xi)dg^{-1}_*\mu_o(\xi)\\
&=\int_{\dd X} \vp(\xi)[\frac{dg^{-1}\mu_o}{d\mu_o}(\xi)]^{1-s}\ol{\phi}(g.\xi)d\mu_o(\xi)\\
&=Q(\vp,g^{-1}\phi)\\
\end{align*}
which proves the $G$-invariance.

Eventually if $f\in\text{Ker}[Q]$ then necessarily $f=0$ $\mu_o$-almost everywhere and since $[\mu_o]$ is supported on the limit set of $G$, $\Lambda(G)$, it follows that $f=0$ on $\Lambda(G)$.
\end{proof}

As a consequence the operator:
$$T\in\text{End}_G(\mathcal{F}_s,\mathcal{F}_{1-s})\mapsto Q(T[\bullet],\bullet)\in\text{Bil}_G(\mathcal{F}_s)$$
 that maps intertwiner between $\mathcal{F}_s$ and $\mathcal{F}_{1-s}$ to $G$-invariant bilinear forms over $\mathcal{F}_s$ is well defined.

\begin{exam}\label{exam:int}
For $s=\half+i\ga$ with $\ga\in\BR$, ${\bf I}_{\mathcal{F}_{\half+i\ga}}\in\text{End}_G(\mathcal{F}_{\half+i\ga},\mathcal{F}_{\half+i\ga})$ and $Q$ itself
 induces a unitary structure on $\mathcal{F}_{\half+i\ga}$ that corresponds to the usual spherical parabolic induction by a unitary character over $A$. We call such type of representations (generalized) principal series.\\
For $s=1$ the projector:
$$T:\mathcal{F}_1\rightarrow \mathcal{F}_{0};\quad \vp\mapsto[\int_{\dd X}\vp(\xi)d\mu_o(\xi)]{\bf1}(\bullet)$$
 is an element of $\text{End}_G(\mathcal{F}_1,\mathcal{F}_{0})$ and 
 $$Q(T[\vp],\psi)=\int_{\dd X}\vp(\xi)d\mu_o(\xi)\ol{\int_{\dd X}\psi(\xi)d\mu_o(\xi)}$$
 induces the trivial representation of $G$.
 \end{exam}

\subsubsection{The $s$-Poisson transform}$ $\\

Analogously to the Poisson transform associated to continuous functions over $\dd X$:
$$x\in X\mapsto\int_{\dd X}\vp(\xi)d\mu_x(\xi)=\int_{\dd X}\vp(\xi)[\frac{d\mu_x}{d\mu_o} ](\xi)d\mu_o(\xi)$$
for $\vp\in\cC(\dd X)$,
it is natural to define the $s$-Poisson transform for $0\le s\le1$ as:
$$\gTh_s[\vp](x)=\int_{\dd X}\vp(\xi)[\frac{d\mu_x}{d\mu_o}]^{1-s}(\xi)d\mu_o(\xi)$$
for $\vp\in\mathcal{F}_s\simeq \cC(\dd X)$ and $x\in X$.
From this point of view the standard transform corresponds to the parameter $s=0$.

The Martin-Poisson correspondence states that the map $\gTh_0:\cC(\dd X)\rightarrow L^\infty(X)$ is injective and given any continuous function $\vp\in\cC(\dd X)$ and geodesic $c$ in $(X,d)$ with $c(0)=o$ and $c(\infty)=\xi\in\Lambda (G)\subset\dd X$:
$$\gTh_0[\vp](c(t))\sim \vp(\xi)=\vp(\xi)\gTh_0[{\bf1}](c(t))$$
when $t$ goes to infinity \cite{MR2346273}.\\
In our investigation on complementary series we shall prove an extension of this correspondence to general $s$-Poisson transforms with $s\neq\half$, namely:
$$\gTh_s[\vp](c(t))\sim \mathcal{T}_s[\vp](\xi)\gTh_s[{\bf1}](c(t))$$
for large $t$ and any $\vp\in\mathcal{F}_s$, where $\mathcal{T}_s$ stands for an operator over $\cC(\dd X)$.

\subsubsection{Intertwiners and asymptotics of $s$-Martin-Poisson transforms with parameters $\half<s\le1$ .}$ $\\
\label{sub:int}

In this subsection we introduce the intertwiner $\mathcal{I}_s$, $\half< s\le1$, between $\mathcal{F}_s$ and $\mathcal{F}_{1-s}$ which plays a central role in the unitarization of those representations and discuss its relation with the $s$-Martin-Poisson transform.

\begin{lem}\label{lem:intcond}
For any $\half<s\le1$ and $\e>0$ one can find $\tau=\tau(\e,s)>0$ such that:
$$\int_{\{\tau\le\langle x,\bullet\rangle_o\}}\frac{1}{d_{o}^{2(1-s)\gd}(x,y)}d\mu_{o,t}(y)\le\e$$
for all $x\in \ol{X}$ and $\gd<t\le T$ as Lemma \ref{lem:equi1}.\\
\end{lem}

\begin{proof}
Let $r_0$ and $T$ as in Lemma \ref{lem:equi1}.
Using the layer cake representation:
\begin{align*}
&\int_{\tau\le\langle x,\bullet\rangle_o}e^{2(1-s)\gd\langle x,y\rangle_{o}}d\mu_{o,t}(y)\\
&=\int_{\ol{X}}e^{2(1-s)\gd\langle x,y\rangle_{o}}d\mu_{o,t}(y)
-\int_{\{\langle x,\bullet\rangle_o<\tau\}}e^{2(1-s)\gd\langle x,y\rangle_{o}}d\mu_{o,t}(y)\\
&=2(1-s)\gd \int_\tau^{+\infty}e^{2(1-s)\gd u }\mu_{o,t}(\{\langle x, \bullet\rangle_{o}>u\})du\\
&\preceq 2(1-s)\gd \int_\tau^{+\infty}e^{2(1-s)\gd u }e^{-\gd u}du\\
&=2(1-s)\gd \int_\tau^{+\infty}e^{(1-2s)\gd u }du=[\frac{1}{(2s-1)}-1]e^{(1-2s)\gd \tau }
\end{align*}
for all $x\in\ol{X}$, $\tau\ge r_0$ and $\gd<t\le T$ which concludes the proof.
\end{proof}

As a consequence the operator:
$$\mathcal{I}_s:\cC(\dd X)\rightarrow \cC({\ol{X}});\quad \vp\mapsto \mathcal{I}_s[\vp](x)=\int_{\dd X}\frac{\vp(\eta)}{d_{o}^{2(1-s)\gd}(x,\eta)}d\mu_o(\eta)$$
is well defined for any $s\in\BR$ with $s> \half$.

\begin{prop}\label{prop:ins}
The operator $\mathcal{I}_s$ is bounded for the uniform norm over $\cC(\dd X)$ and induces an intertwiner between $\mathcal{F}_s$ and $\mathcal{F}_{1-s}$.
In particular the bilinear form over $\mathcal{F}_s$:
$$Q_s(\vp,\psi)=Q(\mathcal{I}_s[\vp],\psi)=\int_{\dd X\times\dd X}\frac{\vp(\xi)\ol{\phi}(\eta)}{d_{o}^{2(1-s)\gd}(\xi,\eta)}d\mu_o\otimes\mu_o(\xi,\eta)$$
is $G$-invariant.
\end{prop}
\begin{proof}
Let $k_1\in\cC(\BR_+)$ be a positive function such that $k_1|_{[0,1]}=1$ and $k_1|_{[2,+\infty)}=0$. Let us denote $k_\tau\in\cC(\BR_+)$ with $\tau>0$ the functions defined as $k_\tau(t)=k_1(\frac{t}{\tau})$ for $t\in\BR_+$.\\
Then the kernel , $K_\tau$, on $\ol{X}\times\dd X$ defined as $(x,\eta)\mapsto k_\tau(e^{2(1-s)\gd\langle x,\eta\rangle_o})e^{2(1-s)\gd\langle x,\eta\rangle_o}$ is continuous and bounded.
In particular the function:
$$\ol{X}\rightarrow \BC;\quad x\mapsto \mathcal{I}_{s,\tau}[\vp](x)=\int_{\dd X}\vp(\eta)K_\tau(x,\eta)d\mu_o(\eta)$$
is continuous for any $\tau\ge0$ .\\

On the other hand using Lemma \ref{lem:intcond} one has:
\begin{align*}
&|\mathcal{I}_{s,\tau}(\vp)(x)-\mathcal{I}_{s,\tau'}(\vp)(x)|\le\|\vp\|_\infty\int_{\min\{\tau,\tau'\}\le\langle x,\bullet\rangle_o}e^{2(1-s)\gd\langle x,\eta\rangle_o}d\mu_o(\eta)
\end{align*}
which goes to zero when $\tau$ and $\tau'$ go to infinity.
In other words $(\mathcal{I}_{s,\tau}(\vp))_\tau$ is Cauchy for the uniform norm over $\ol{X}$ and its limit, $\mathcal{J}_s[\vp]$, is therefore continuous.

Since $K_\tau(x,\bullet)$ is dominated by the integrable function $\eta \mapsto e^{2(1-s)\gd\langle x,\eta\rangle_o}$ for any fixed $x\in\ol{X}$ the dominated convergence theorem implies that: 
$$\mathcal{J}_s[\vp](x)=\mathcal{I}_{s}[\vp](x)$$
for all $x\in\ol{X}$ and thus $\mathcal{I}_{s}$ is well defined and bounded.\\

Eventually for any $f\in\mathcal{F}_s$, $g\in G$ one has:
\begin{align*}
\mathcal{I}_s&[\pi_s(g)\vp](\eta)=\int_{\dd X} [\frac{dg\mu_o}{d\mu_o}(\xi)]^{s}\frac{\vp(g^{-1}{\xi})}{d_{o}^{2(1-s)\gd}(\xi,\eta)}d\mu_o({\xi})\\
&=\int_{\dd X} e^{-s\gd b_{\xi}(g.o,o)}\frac{\vp(g^{-1}\xi)}{d_{o}^{2(1-s)\gd}(\eta,\xi)}d\mu_o(\xi)\\
&=\int_{\dd X} e^{-s\gd b_{g.\xi}(g.o,o)}\frac{\vp(\xi)}{d_{o}^{2(1-s)\gd}(\eta,g.\xi)}dg^{-1}_*\mu_o(\xi)\\
&=\int_{\dd X} e^{-s\gd b_{\xi}(o,g^{-1}.o)}\frac{\vp(\xi)}{d_{o}^{2(1-s)\gd}(g^{-1}.\eta,\xi)}e^{(1-s)\gd[b_{\xi}(g^{-1}.o,o)+b_{\eta}(o,g.o)]}e^{-\gd b_{\xi}(g^{-1}.o,o)}d\mu_o(\xi)\\
&=e^{(1-s)\gd b_{\eta}(o,g.o)}\int_{\dd X} \frac{\vp(\xi)}{d_{o}^{2(1-s)\gd}(g^{-1}\eta,\xi)}d\mu_o(\xi)=\pi_{1-s}(g)\mathcal{I}_s[\vp](\eta)
\end{align*}
for $[\mu_o]$-almost every $\eta$.
\end{proof}

\begin{lem}\label{lem:issurl2}
The operator $\mathcal{I}_s$ extends to $L^2(\dd X,\mu_o)$ as a bounded self-adjoint operator.
\end{lem}
\begin{proof}
The Cauchy-Schwarz inequality implies:
\begin{align*}
|\mathcal{I}_s(\vp)|^2(x)&=|\int_{\dd X}\vp(\xi)e^{2(1-s)\gd\langle\xi,x\rangle_o}d\mu_o(\xi)|^2\\
&\le \int_{\dd X}|\vp|^2(\xi)e^{2(1-s)\gd\langle\xi,x\rangle_o}d\xi.\int e^{2(1-s)\gd\langle\xi,x\rangle_o}d\mu_o(\xi)\\
&\prec_s\int_{\dd X}|\vp|^2(\xi)e^{2(1-s)\gd\langle\xi,x\rangle_o}d\mu_o(\xi)=\mathcal{I}_s(|\vp|^2)(x)
\end{align*}
for all $\vp\in L^2(\dd X,\mu_o)$ and $x\in\ol{X}$.
It follows that:
\begin{align*}
\|\mathcal{I}_s(\vp)\|^2&\prec_s\|\mathcal{I}_s(|\vp|^2)\|^2=\int_{\dd X}|\vp|^2(\xi)[\int_{\dd X}e^{2(1-s)\gd\langle\xi,\eta\rangle_o}d\mu_o(\eta)]d\mu_o(\xi)\\
&\prec_s\|\vp\|_2^2
\end{align*}
 for all $\vp\in L^2(\dd X,\mu_o)$.
 The fact that $\mathcal{I}_s$ is self-adjoint follows from its definition.
\end{proof}
Observe that the quadratic form $Q_s$ associated to $\mathcal{I}_s$ is positive if and only if $\mathcal{I}_s$ as a $L^2$-operator is positive.
Moreover the Banach representations $(\pi_s)_{0\le s\le1}$ over $\cC(\dd X)$ extend over $L^2(\dd X,\mu_o)$ as bounded operators and Lemma \ref{lem:dual} \ref{lem:issurl2} with Proposition \ref{prop:ins} imply $\pi_s(g)^*=\pi_{1-s}(g^{-1})$ and $\pi_{1-s}(g)\mathcal{I}_s=\mathcal{I}_s\pi_s(g)$ for all $g\in G$.\\

We conclude this subsection with the following observations:
\begin{lem}\label{lem:poies}
For any  $0< s<1$ with $s\neq\half$:
$$\gTh_{s}[{\bf1}](x)\asymp e^{-(\half+|s-\half|)\gd d(x,o)}+\frac{2(\half+|s-\half|).e^{-\half\gd d(x,o)}}{(s-\half)}\sinh[(s-\half)\gd d(x,o)]$$
for $m$-almost every $x\in X$.
In particular:
$$\gTh_{s}[{\bf1}](x)\sim e^{-(\half-|s-\half|)\gd d(x,o)}.$$
when $d(x,o)$ goes to infinity.
\end{lem}
\begin{proof}
Since $m$ is supported on the orbit $G.o$ and $\gTh_s[{\bf1}](g.o)=\gTh_{1-s}[{\bf1}](g^{-1}.o)$ for all $s\in\BR$ and $g\in G$. 
It is therefore enough to estimate $\gTh_{1-s}[{\bf1}]$ for $\half< s<1$:
\begin{align*}
&\gTh_{1-s}[{\bf1}](x)=\int_{\dd X}e^{-s\gd b_\xi(x,o)}d\mu_o(\xi)\\
&=\int_{\dd X}e^{-s\gd[d(x,o)-2\langle\xi,x\rangle_o]}d\mu_o(\xi)\\
&=e^{-s\gd d(x,o)}\int_{\dd X}e^{2s\gd\langle\xi,x\rangle_o}d\mu_o(\xi)\\
&=e^{-s\gd d(x,o)}[1+2s\gd\int_0^{d(x,o)}e^{2s\gd t}\mu_o(\{\langle x,\bullet\rangle_o\ge t\})dt]\\
&\asymp e^{-s\gd d(x,o)}[1+2s\gd\int_0^{d(x,o)}e^{(2s-1)\gd t}dt]
\end{align*}
for $m$-almost every $x\in X$.
It follows:
\begin{align*}
&\gTh_{1-s}[{\bf1}](x)=e^{-s\gd d(x,o)}+e^{-s\gd d(x,o)}2s\gd\int_0^{d(x,o)}e^{(2s-1)\gd t}dt\\
&=e^{-(\half+|s-\half|)\gd d(x,o)}+\frac{2(\half+|s-\half|).e^{-\half\gd d(x,o)}}{(s-\half)}\sinh[(s-\half)\gd d(x,o)]
\end{align*}
\end{proof}

\begin{rem}
Using an similar approach one can prove that for $0\le s\le1$ with $s\neq\half$, the operator $\gTh_s$ takes values in $L^p(X,m)$ if and only if  $\half<s\le1$ and $s< 1-\frac{1}{p}$ or $0\le s<\half$ and $\frac{1}{p}< s$.
In such cases $\gTh_s$ defines an intertwiner between the representation $\mathcal{F}_s$ and Koopman representation on $L^p(X,m)$.
\end{rem}

It follows from Lemma \ref{lem:poies}:
\begin{align*}
\gTh_{s}[\vp](x)&=\int_{\dd X}\vp(\xi)e^{(1-s)\gd b_\xi(x,o)}d\mu_o(\xi)\\
&=e^{-(1-s)\gd d(x,o)}\int_{\dd X}\vp(\xi)e^{2(1-s)\gd \langle \xi,x\rangle_o}d\mu_o(\xi)\\
&=e^{-(1-s)\gd d(x,o)}\mathcal{I}_s[\vp](x)\asymp\mathcal{I}_s[\vp]\gTh_s({\bf1})
\end{align*}
for any $\half<s\le1$, $\vp\in\cC(\dd X)\simeq\mathcal{F}_s$ and $x\in \ol{X}$.
In other words $\mathcal{I}_s[\vp]$ extends the asymptotic part of the Martin-Poisson correspondence for $\half<s\le1$.
The injectivity of $\gTh_s$ is discussed in Section \ref{sec:end}.


\section{The unitary structures over the spaces of homogeneous functions} \label{sec:pos}$ $\\
This section is concerned with the positivity of the $G$-invariant quadratic forms $Q_s$ over $\mathcal{F}_{s}$ associated to the operators $\mathcal{I}_s$ introduced in subsection \ref{sub:int} for $\half<s\le1$ assuming the distance over $X$ is conditionally negative.\\
In a second time we use the duality between $\mathcal{F}_s$ and $\mathcal{F}_{1-s}$ for $\half<s\le1$ to deduce a unitary structure over $\mathcal{F}_{1-s}$ whenever $\mathcal{F}_{s}$ is unitarizable.
Together with Example \ref{exam:int} this will prove Theorem 1.\\

Let $\half<s\le1$ and consider the measures $(m_{o,s,t})_{t>\gd}$ defined as:
$$dm_{o,s,t}(x,y)=\frac{1}{d_{o}^{2(1-s)\gd}(x,y)}d\mu_{o,t}(x)d\mu_{o,t}(y)$$
on $\ol{X}\times\ol{X}$.

\begin{cor}\label{cor:mscv}
Up to a subsequence, the finite measures $(m_{o,s,t})_{t>\gd}$ converge weakly to 
$$dm_{o,s}(\xi,\eta)=e^{2(1-s)\gd\langle \xi,\eta\rangle_{o}}d\mu_{o}(\xi)d\mu_{o}(\eta)$$
supported over $\Lambda(G)\times\Lambda(G)$ when $t$ goes to $\gd$.
\end{cor}
\begin{proof}
Assume $(\mu_{o,t})_t$ converges weakly to $\mu_{o}$ over $\ol{X}$.

Let $k_\tau\in\cC(\BR_+)$, with $\tau>0$, be a positive function such that $k_\tau|_{[0,\tau]}=1$ and $k_\tau|_{[\tau+1,+\infty)}=0$ and 
$$K_\tau:\ol{X}\times\ol{X}\rightarrow\BR_+;\quad (x,y)\mapsto k_\tau(e^{2(1-s)\gd\langle x,y\rangle_o})e^{2(1-s)\gd\langle x,y\rangle_o}$$ 
as in Proposition \ref{prop:ins} proof.
Since the kernel $K_\tau$ is continuous on $\ol{X}\times\ol{X}$ and $(\mu_{o,t})_t$ converges weakly to $\mu_{o}$, for any fixed $\tau>0$ and any continuous function $\Phi$ over $\ol{X}\times\ol{X}$:
$$|\int_{\ol{X}\times\ol{X}}\Phi(x,y)K_\tau(x,y)d\mu_{o,t}\otimes\mu_{o,t}(x,y)-\int_{\dd X\times\dd X}\Phi(\xi,\eta)K_\tau(\xi,\eta)d\mu_{o}\otimes\mu_{o}(\xi,\eta)|$$
converges to $0$ when $t$ goes to $\gd$.

On the other hand Lemma \ref{lem:intcond} implies that for any $\e>0$ one can find $\tau_\e>0$ and $\gd<T$ such that $m_{o,s,t}(\{\langle \bullet,\bullet\rangle_{o}\ge \tau_\e\}),m_{o,s}(\{\langle \bullet,\bullet\rangle_{o}\ge \tau_\e\})\le \e$ for any $\gd<t\le T$.
It follows:
\begin{align*}
&|\int_{\ol{X}\times\ol{X}}\Phi(x,y)[e^{2(1-s)\gd\langle x,y\rangle_{o}}-K_{\tau_\e}(x,y)]d\mu_{o,t}(x)d\mu_{o,t}(y)\\
&-\int_{\dd X\times\dd X}\Phi(\xi,\eta)[e^{2(1-s)\gd\langle \xi,\eta\rangle_{o}}-K_{\tau_\e}(\xi,\eta)]d\mu_{o}(\xi)d\mu_{o}(\eta)|\\
&\le|\int_{\{\langle \bullet,\bullet\rangle_{o}\ge \tau_\e\}\cap\ol{X}\times\ol{X}}\Phi(x,y)e^{2(1-s)\gd\langle x,y\rangle_{o}}d\mu_{o,t}(x)d\mu_{o,t}(y)\\
&-\int_{\{\langle \bullet,\bullet\rangle_{o}\ge \tau_\e\}\cap\dd {X}\times\dd {X}}\Phi(\xi,\eta)e^{2(1-s)\gd\langle \xi,\eta\rangle_{o}}d\mu_{o}(\xi)d\mu_{o}(\eta)|\\
&\le2\|\Phi\|_\infty\e
\end{align*}
which concludes the proof
\end{proof}

Let us recall some materials concerning kernels over topological spaces:
\begin{defn}
A continuous kernel $q$ over $X$ is called positive definite if there exist a Hilbert space ${\bf H}$ and a continuous map $r:X\rightarrow {\bf H}$ such that:
$$q(x,y)=(r(x),r(y))_{\bf H}$$
for all $x,y\in X$.
\end{defn}
Equivalently a continuous kernel $q$ over $X$ is positive definite if for all finite set $F\subset X$ and any family of complex numbers $(c_x)_{x\in F}$:
$$\sum_{x,y\in F}c_x\ol{c_y}q(x,y)$$
is positive \cite{MR2415834}.

\begin{thm*}[Schoenberg's]
Let $k$ be a conditionally negative kernel over $X$.
Then for any $t\ge0$, the kernel $q_t=e^{-tk}$ is positive definite over $X$.
\end{thm*}

The distance of a negatively curved $G$-space $(Z,d_Z)$ is called \textit{roughly conditionally negative} if it decomposes as $d_Z=N_0+\upsilon$ over a $G$-subspace $Z'\subset Z$ with $\mathcal{Q}(\Lambda(G))\subset Z'$ where:
\begin{itemize}
\item $N_0$ is a $G$-invariant conditionally negative kernel over $Z'$;
\item $\upsilon$ a continuous kernel over $Z'$, for all $o'\in Z'/G$, there exists $\k(G.o')\in\BR$ such that $\upsilon|_{G.o'\times G.o'}$ satisfies 
$$\upsilon|_{G.o'\times G.o'}(z,z')\xrightarrow{d(z,z')\rightarrow+\infty} \k(G.o')\in\BR$$
\end{itemize}
\begin{rem}\label{rem:reduction1}
In our case, because
$$\langle z,z'\rangle_o=\half(N_0(z,o)+N_0(z',o)-N_0(z,z'))+\half(\upsilon(z,o)+\upsilon(z',o)-\upsilon(z,z'))$$
for all $z,z'\in G.o$, the product 
$$(z,z')\in G.o\times G.o\mapsto\langle z,z'\rangle_{N_0,o}=\half(N_0(z,o)+N_0(z',o)-N_0(z,z'))$$ 
extends continuously to $\Lambda(G)\times \Lambda(G)$ as 
$$\langle \xi,\eta\rangle_{N_0,o}=\langle \xi,\eta\rangle_{o}-\half\k(G.o)$$
for all $\xi,\eta\in \Lambda(G)$.\\
Moreover the measures $(m_{o,s,t})_{t>\gd}$ supported over $G.o\times G.o\subset \ol{X}\times\ol{X}$ satisfy:
\begin{align*}
&dm_{o,s,t}(x,y)=\frac{1}{d_{o}^{2(1-s)\gd}(x,y)}d\mu_{o,t}(x)d\mu_{o,t}(y)\\
&=e^{(1-s)\gd(\upsilon(z,o)+\upsilon(z',o)-\upsilon(z,z'))}e^{2(1-s)\gd\langle x,y\rangle_{N_0,o}}d\mu_{o,t}(x)d\mu_{o,t}(y)\\
\end{align*}
and thus, using Corollary \ref{cor:mscv}, we have:
$$e^{2(1-s)\gd\langle \xi,\eta\rangle_{N_0,o}}d\mu_{o}(\xi)d\mu_{o}(\eta)=e^{-(1-s)\gd\k(G.o)}dm_{o,s}(\xi,\eta)$$
\end{rem}

\begin{prop}\label{prop:ispos}
Assume the distance $d$ over $X$ is roughly conditionally negative.
Then for any $\half<s\le1$ the quadratic form $Q_s$ over $\mathcal{F}_s$ constructed from the intertwiner $\mathcal{I}_s$ is positive or equivalently $\mathcal{I}_s$ as an operator over $L^2(\dd X,\mu_o)$ is positive.
\end{prop}

\begin{proof}
According to Corollary \ref{cor:mscv} one can assume that $(m_{o,s,t})_t$ converges weakly to $m_{o,s}$.
Moreover, using Remark \ref{rem:reduction1}, we can assume the distance $d$ to be conditionally negative.

The Schoenberg theorem implies that for any $0\le s\le1$ one can find a Hilbert space, ${\bf H}_{1-s}$, and a continuous map $r_{1-s}:X\rightarrow {\bf H}_{1-s}$ such that:
$$e^{-(1-s)\gd d(x,y)}=(r_{1-s}(x),r_{1-s}(y))_{1-s}$$
for all $x,y\in X$.
For any $\vp\in\cC(\ol{X})$ it follows that :
\begin{align*}
&q_{s,t}(\vp):=\int_{\ol{X}\times\ol{X}}\vp(x)\ol{\vp(y)}e^{2(1-s)\gd\langle x,y\rangle_{o}}d\mu_{o,t}(x)d\mu_{o,t}(y)\\
&=\int_{X\times X}\vp(x)\ol{\vp(y)}(e^{(1-s)\gd d(x,o)}r_{1-s}(x),e^{(1-s)\gd d(y,o)}r_{1-s}(y))_{1-s}d\mu_{o,t}(x)d\mu_{o,t}(y)
\end{align*}
for all $\gd<t$.\\

For $t>\gd$ fixed a dominated convergence argument shows that given an exhaustion of compact sets $(K_n)_n$, i.e. a increasing sequence of compact sets, $K_n$, with $\bigcup_n K_n=X$ and $K_n\subset\mathring{K_{n+1}}$ one has:
$$\int_{K_n\times K_n}\vp(x)\ol{\vp(y)}(e^{(1-s)\gd d(x,o)}r_{1-s}(x),e^{(1-s)\gd d(y,o)}r_{1-s}(y))_{1-s}d\mu_{o,t}(x)d\mu_{o,t}(y)\rightarrow q_{s,t}(\vp)$$
when $n$ goes to infinity.\\

Since $K_n$ is compact the linear form:
$$L_{n}:v\mapsto \int_{K_n}\ol{\vp}(x)(v,e^{(1-s)\gd d(x,o)}r_{1-s}(x))_{1-s}d\mu_{o,t}(x)$$
over ${\bf H}_{1-s}$ is bounded with norm $\|L_{n}\|\le \|\vp\|_\infty e^{\gd\text{diam}(\{o\}\cup K_n)}$. 
We denote:
$$\int_{K_n}\vp(x)e^{(1-s)\gd d(x,o)}r_{1-s}(x)\in {\bf H}_{1-s}$$
the unique vector such that
$$L_{n}(v)=(v,\int_{K_n}\vp(x)e^{(1-s)\gd d(x,o)}r_{1-s}(x))_{1-s}$$

In particular:
\begin{align*}
&\int_{K_n\times K_n}\vp(x)\ol{\vp(y)}(e^{(1-s)\gd d(x,o)}r_{1-s}(x),e^{(1-s)\gd d(y,o)}r_{1-s}(y))_{1-s}d\mu_{o,t}(x)d\mu_{o,t}(y)\\
&=(\int_{K_n}\vp(x)e^{(1-s)\gd d(x,o)}r_{1-s}(x)d\mu_{o,t}(x),\int_{K_n}\vp(y)e^{(1-s)\gd d(x,o)}r_{1-s}(y)d\mu_{o,t}(y))_{1-s}\ge0
\end{align*}
and therefore $q_{s,t}(\vp)\ge0$ for all $t>\gd$.

Eventually using Corollary \ref{cor:mscv} one has:
\begin{align*}
&q_{s,t}(\vp)\xrightarrow{t\rightarrow \gd}\int_{\dd X\times\dd X}\vp(\xi)\ol{\vp(\eta)}e^{2(1-s)\gd\langle \xi,\eta\rangle_{o}}d\mu_o(\xi)d\mu_o(\eta)=Q(\mathcal{I}_s(\vp),\vp)
\end{align*}
which is positive as a limit of the positive sequence $(q_{s,t}(\vp))_{t>\gd}$.
\end{proof}


Let $\mathcal{F}'_{1-s}$, with $\half< s\le 1$, be the sub-representation $\mathcal{F}'_{1-s}=\mathcal{I}_s[\mathcal{F}_{s}]\subset \mathcal{F}_{1-s}$.

\begin{cor}
For any $\half<s\le1$, there exists a unitary structure over $\mathcal{F}'_{1-s}$ such that the intertwiner $\mathcal{I}_s:\mathcal{F}_s\rightarrow\mathcal{F}_{1-s}$ extends as an isometric intertwiner between $\pi_{s}$ and $\pi_{1-s}$.
In addition $Q$ extends uniquely as a coupling between $\mathcal{H}_s$ and  $\mathcal{H}_{1-s}$, which denote respectively the Hilbert completion of $\mathcal{F}_{s}$ and $\mathcal{F}_{1-s}$, such that:
$$|Q(v,w)|\le\|v\|_s\|w\|_{1-s}$$
where $v\in\mathcal{H}_s$ and $w\in\mathcal{H}_{1-s}$.
\end{cor}
\begin{proof}
Using the self-adjointness of $\mathcal{I}_s$ over $L^2(\dd X,\mu_o)$ it appears that the bilinear form  over $\mathcal{F}'_{1-s}\subset L^2(\dd X,\mu_o)$, $Q_{1-s}$, given as:
$$Q_{1-s}(\vp',\phi')=(\mathcal{I}_s[\vp],\phi)$$
for $\vp'=\mathcal{I}_s[\vp],\phi'=\mathcal{I}_s[\phi]\in \mathcal{F}'_{1-s}$ with $\vp,\phi\in\mathcal{F}_{s}$ is well defined.

It follows
$$Q_{1-s}(\vp')=(\mathcal{I}_s[\vp],\vp)\ge0$$
since $\mathcal{I}_s$ is positive and
$$Q_{1-s}(\pi_{1-s}(g)\vp')=(\mathcal{I}_s\pi_{s}(g)\vp,\pi_{s}(g)\vp)=(\mathcal{I}_s[\vp],\vp)=Q_{1-s}(\vp')$$
for all $\vp'=\mathcal{I}_s[\vp]\in\mathcal{F}'_{1-s}$ with $\vp\in\mathcal{F}_{s}$ and $g\in G$.

By the very definition of $Q_{1-s}$ the operator $\mathcal{I}_s:\mathcal{F}_s\rightarrow \mathcal{F}'_{1-s}$ is isometric and extends to an unitary intertwiner between $\mathcal{H}_{s}$ and $\mathcal{H}_{1-s}$.

Moreover the Cauchy-Schwarz inequality implies:
$$|Q(\vp,\mathcal{I}_s[\phi])|\le\|\vp\|_s\|\phi\|_s= \|\vp\|_s\|\phi'\|_{1-s}$$
for all $\vp,\phi\in\mathcal{F}_s$ and $\phi'=\mathcal{I}_s[\phi]\in\mathcal{F}_{1-s}$ and thus $Q$ extends as a coupling between $\mathcal{H}_{s}$ and $\mathcal{H}_{1-s}$.
Lemma \ref{lem:dual} implies that $Q$ is non-degenerated over $\mathcal{H}_{s}\times\mathcal{H}_{1-s}$.
\end{proof}

We extend the family of representations $\mathcal{H}_{s\in[0,1]\setminus\{\half\}}$ at $\half$ with the Koopman representation $\mathcal{H}_\half=L^2(\dd X,\mu_o)\subset L^2(\dd X,\mu_o)$, that is nothing but the so-called boundary representation \cite{MR2787597}. 
In Proposition \ref{prop:fell} we prove this extension is actually continuous for the Fell topology over the unitary dual of $G$.

In the rest we refer to the representations $(\mathcal{H}_s,\pi_s)_{s\in[0,1]}$ as complementary even if the proof of Theorem 2 will only be given in Section \ref{sec:end}.

\section{Analysis of matrix coefficients}\label{sec:anamat}
From this section the operators $\mathcal{I}_s$ over $L^2(\dd X,\mu_o)$ is assumed to be positive for $\half<s\le1$.
According to Proposition \ref{prop:ispos} this is the case whenever the distance over $X$ is conditionally negative.

This section is dedicated to estimates over averages of matrix coefficients associated to unitary representations $(\mathcal{H}_s)_{s\in[0,1]}$ introduced in Section \ref{sec:pos}.
This will be our principal tool in the investigation of those representations in Section \ref{sec:end}.\\

We start by extending the asymptotic estimates of the $s$-Martin-Poisson transforms discussed in Subsection \ref{sub:int} for parameter $\half<s\le1$ to $0\le s\le1$.

\begin{lem}\label{lem:extmart}
Given $\half\le s\le1$ and $\vp\in \cC(\dd X)$, there exists a positive decreasing function $\go_\vp$ with $\go_\vp(t)\xrightarrow{t\rightarrow+\infty}0$ such that:
$$\frac{\gTh_{1-s}[|\vp-\vp(\hat{x})|]}{\gTh_{1-s}[{\bf1}]}(x)\le \go_\vp(d(o,x))$$
for all $x\in \mathcal{Q}(\Lambda(G))$.
In particular 
$$\frac{\gTh_{1-s}[\vp]}{\gTh_{1-s}[{\bf1}]}(x)=\frac{1}{\gTh_{1-s}[{\bf1}](x)}\int_{\dd X}[\frac{d\mu_x}{d\mu_o}]^s(\xi)\vp(\xi)d\mu_o(\xi)\xrightarrow{x\rightarrow\eta}\vp(\eta)$$
for all $\eta\in \Lambda(G)$, i.e. $\frac{\gTh_{1-s}[\vp]}{\gTh_{1-s}[{\bf1}]}$ extends continuously over the closed geodesic hull of the limit set, $\ol{\mathcal{Q}(\Lambda(G))}$.
\end{lem}

\begin{proof}
The case $s=1$ corresponds to the classic Martin-Poisson transform and $s=\half$ is \cite{Garncarek:2014aa} Lemma 5.3 proof. 
One can therefore assume $\half< s<1$.
Given $\vp\in \cC(\dd X)$ and $x\in \mathcal{Q}(\Lambda(G))$ we denote $\nabla_x\vp(\xi)=\vp(\xi)-\vp(\hat{x})$.

It is enough to prove that for any $\e$, there exists $\rho$ such that $d(o,x)\ge \rho$ implies
$$\frac{\gTh_{1-s}[|\nabla_x\vp|]}{\gTh_{1-s}[{\bf1}]}(x)\le\e$$
Because $\hat{\vp}$ is uniformly continuous over $\dd X$ there exists a bounded positive decreasing function over $\BR_+$  with $\go_\vp'(t)\xrightarrow{t\rightarrow+\infty}0$ such that:
$$|\nabla_x\vp|(\xi)\le\go_\vp'(\langle\xi,\hat{x}\rangle_o)\le \go_\vp'(\langle\xi,x\rangle_o-C)$$
for some universal constant $C\ge0$.\\
On the other hand for any $\rho'>r$:
\begin{align*}
\frac{1}{\gTh_{1-s}[{\bf1}](x)}\int_{\mathcal{O}_o(x,d(o,x)-\rho')^c} e^{-s\gd b_\xi(x,o)}&d\mu_o(\xi)=
\frac{e^{s\gd d(x,o)}}{\gTh_{1-s}[{\bf1}](x)}\int_{\mathcal{O}_o(x,d(o,x)-\rho')^c} e^{2s\gd\langle\xi,x\rangle}d\mu_o(\xi)\\
&\asymp e^{(1-2s)\gd d(x,o)}\int_{\mathcal{O}_o(x,d(o,x)-\rho')^c} e^{2s\gd\langle\xi,x\rangle}d\mu_o(\xi)\\
&\le e^{(1-2s)\gd d(x,o)}e^{2s\gd \rho'}
\end{align*}
Eventually given $\e$ and $\rho>\rho'\ge r$ such that $\go_\vp'(\rho'-C)\le\half\e$ and $e^{-(1-2s)\gd \rho}e^{2s\gd \rho'}\le\frac{\e}{4\|\vp\|_\infty}$ one has:
\begin{align*}
\frac{\gTh_{1-s}[|\nabla_x\vp|]}{\gTh_{1-s}[{\bf1}]}&(x)=\frac{1}{\gTh_{1-s}[{\bf1}](x)}\int_{\dd X} \gTh_{1-s}[|\nabla_x\vp|](\xi)e^{-s\gd b_\xi(x,o)}d\mu_o(\xi)\\
&=\frac{1}{\gTh_{1-s}[{\bf1}](x)}\int_{\mathcal{O}_o(x,d(o,x)-\rho')} |\nabla_x\vp|(\xi)e^{-s\gd b_\xi(x,o)}d\mu_o(\xi)\\
&+\frac{1}{\gTh_{1-s}[{\bf1}](x)}\int_{\mathcal{O}_o(x,d(o,x)-\rho')^c} |\nabla_x\vp|(\xi)e^{-s\gd b_\xi(x,o)}d\mu_o(\xi)\\
&\le\go_\vp'(\rho'-C)+2\|\vp\|_\infty e^{-(1-2s)\gd d(x,o)}e^{2s\gd \rho'}\le\e
\end{align*}
whenever $d(o,x)\ge\rho$.

\end{proof}

We are ready to prove the following equidistribution over \textit{regular} matrix coefficients:
\begin{prop}\label{prop:matequi}
Given any continuous functions $\vp, \phi\in \cC(\dd X)$ and $f_1,f_2\in\cC(\ol{X})$ the averages of matrix coefficients:
$$\int_{G} \frac{(\pi_s(g)\vp,\phi)}{\gTh_{1-s}[{\bf1}](g.o)}.f_1(g.o)f_2(g^{-1}.o)d\nu_{o,t}(g)$$
equidistribute when $t$ goes at infinity to:
$$(\mathcal{I}_s[\vp],f_2|_{\dd X})(\phi,f_1|_{\dd X})$$
\end{prop}
As a consequence:
$$\int_{G} \frac{(\mathcal{I}_s[\pi_s(g)\vp],\phi)}{\gTh_{s}[{\bf1}](g^{-1}.o)}.f_1(g.o)f_2(g^{-1}.o)d\nu_{o,t}(g)\rightarrow (\mathcal{I}_s[\vp],f_2|_{\dd X})(\mathcal{I}_s[\phi],f_1|_{\dd X})$$
for all $\vp, \phi\in \cC(\dd X)$, $f_1,f_2\in\cC(\ol{X})$ and $\half\le s\le 1$.

\begin{proof}
Let $t\ge0$ and denote:
$$I_t=|\int_{G} \frac{(\pi_s(g)\vp,\phi)}{\gTh_{1-s}[{\bf1}](g.o)}.f_1(g.o).f_2(g^{-1}.o)-\mathcal{I}_s[\vp](\widecheck{g})\phi(\hat{g}).f_1(g.o).f_2(g^{-1}.o)d\nu_{o,t}(g)|$$
Lemma \ref{lem:vitali} implies:
\begin{align*}
&I_t\le\int_{S_{o,G}^*(t)} |\frac{(\pi_s(g)\vp,\phi)}{\gTh_{1-s}[{\bf1}](g.o)}-\mathcal{I}_s[\vp](\widecheck{g}).\phi(\hat{g})|.|f_1(g.o).f_2(g^{-1}.o)|\mu_{o}({O_{o}^{(2)}(g)})dg\\
&\le\|f_1\|_\infty\|f_2\|_\infty\int_{S_{o,G}^*(t)} a(g).\mu_o({\mathcal{O}_o(g.o,R'(t)))\mu_o(\mathcal{O}_o(g^{-1}.o,R'(t))})dg
\end{align*}
where
$$a(g)=|\frac{(\pi_s(g)\vp,\phi)}{\gTh_{1-s}[{\bf1}](g.o)}-\mathcal{I}_s[\vp](\widecheck{g}).\phi(\hat{g})|$$
for $t\ge0$ and $g\in S_{o,G}^*(t)$.


Let $\go_\phi$ as in Lemma \ref{lem:extmart} and $\go_\vp$ the visual modulus of continuity over the compact set $\mathcal{Q}[\Lambda(G)]$ of $\frac{\gTh_{s}[\vp]}{\gTh_{s}[{\bf1}]}$.

For any $g\in S_{o,G}(t)$ observe that:
\begin{align*}
&a(g)=|\frac{(\pi_s(g)\vp,\phi-\phi(\hat{g}){\bf1})}{\gTh_{1-s}[{\bf1}](g.o)}+\phi(\hat{g})[\frac{\gTh_{s}[\vp](g^{-1}.o)}{\gTh_{s}[{\bf1}](g^{-1}.o)}-\mathcal{I}_s(\vp)(\widecheck{g})]|\\
&\le\|\vp\|_\infty|\frac{\gTh_{1-s}[|\phi-\phi(\hat{g})|]}{\gTh_{1-s}[{\bf1}]}(g.o)|+\|\phi\|_\infty|\frac{\gTh_{s}[\vp]}{\gTh_{s}[{\bf1}]}(g^{-1}.o)-\mathcal{I}_s(\vp)(\widecheck{g})|\\
&\le \|\vp\|_\infty\go_{\phi}(d(o,g.o))+\|\phi\|_\infty\go_{\vp}(\langle g^{-1}.o,\widecheck{g}\rangle_o)\le \|\vp\|_\infty\go_{\phi}(t-R)+\|\phi\|_\infty\go_{\vp}(t-R)
\end{align*}
which goes to $0$ when $t$ goes to infinity.




On the other hand according to Lemma \ref{lem:vitali}:
 $$\mu_o({\mathcal{O}_o(g.o,R'(t)))\mu_o(\mathcal{O}_o(g^{-1}.o,R'(t))})|S_R^*(t)|=e^{2\gd r'}e^{-\gd tR}|S_R^*(t)|\asymp1$$
Eventually the dominated convergence theorem implies:
\begin{align*}
I_t&\le\|f_1\|_\infty\|f_2\|_\infty \int_{S_{o,R}^*(t)}a(g)\mu_o({\mathcal{O}_o(g.o,R'(t)))\mu_o(\mathcal{O}_o(g^{-1}.o,R'(t))})dg\\
&\le[\|f_1\|_\infty\|f_2\|_\infty e^{2\gd r'}]e^{-\gd tR}\int_{S_{o,R}^*(t)}a(g)dg\\
&\le [\|f_1\|_\infty\|f_2\|_\infty e^{2\gd r'}]e^{-\gd tR}|S_{o,R}^*(t)|[\|\vp\|_\infty\go_{\phi}(t-R)+\|\phi\|_\infty\go_{\vp}(t-R)]\rightarrow 0
\end{align*}
\end{proof}

In the following we extend Proposition \ref{prop:matequi} to all $\mathcal{H}_s$-matrix coefficients.

We shall need the following estimate:

\begin{lem}\label{lem:is}
Given any $\half<s\le1$ one has:
$$\mathcal{I}_s({\bf1})\asymp [\frac{1}{(2s-1)}]{\bf1}.$$
\end{lem}
\begin{proof}
For any $\xi\in\dd X$ one has:
\begin{align*}
&\mathcal{I}_s({\bf1})(\xi)=\int_{\dd X}d^{-2(1-s)\gd}_{o}(\xi,\eta)d\mu_o(\eta)=\int_{\dd X}e^{2(1-s)\gd\langle\xi,\eta\rangle_o}d\mu_o(\eta)\\
&=1+2(1-s)\gd\int_0^{+\infty}e^{2(1-s)\gd t}\mu_o(\{\langle\xi,\bullet\rangle_o\ge t\}dt\\
&\asymp 1+2(1-s)\gd\int_0^{+\infty}e^{(1-2s)\gd t}dt= \frac{1}{(2s-1)}
\end{align*}
\end{proof}

\begin{prop}\label{prop:incl}
For any $\half< s\le1$ the following inclusions hold:
$$(\mathcal{H}_{1-s},Q_{1-s})\subset (L^2(\dd X,\mu_o), \|\,\,\|^2_2) \subset (\mathcal{H}_{s},Q_{s})$$
In other words:
$$\|v\|^2_{2}\prec_s Q_{1-s}(v)$$
for any $v\in\mathcal{H}_{1-s}$ and:
$$Q_s(w)\prec_s\|w\|^2_{2}$$
 for any $w\in L^2(\dd X,\mu_o)$.
\end{prop}
\begin{proof}
Let us considere the bounded operator $d_s:L^2(\dd X,\mu_o)\rightarrow L^2(\dd X\times\dd X,\mu_o\otimes\mu_o)$ defined as:
$$d_s[\vp](\xi,\eta)=\frac{\vp(\xi)-\vp(\eta)}{d^{(1-s)\gd}_{o}(\xi,\eta)}$$
for $(\xi,\eta)\in \dd X\times\dd X$ and $\mathfrak{D}_s=\half d_s^*d_s\in\mathcal{B}[L^2(\dd X,\mu_o)]$.

Observe that:
\begin{align*}
\half\|d_s[\vp]\|^2_2=(\mathfrak{D}_s\vp,\vp)&=\half\int_{\dd X^{\times 2}}\frac{|\vp(\xi)-\vp(\eta)|^2}{d_{o}^{2(1-s)\gd}(\xi,\eta)}d\mu_o(\xi)d\mu_o(\eta)\\
&=\int_{\dd X}|\vp(\xi)|^2[\int_{\dd X}\frac{d\mu_o(\eta)}{d_{o}^{2(1-s)\gd}(\xi,\eta)}]d\mu_o(\xi)-\text{Re}[(\mathcal{I}_s[\vp],\vp)]\\
&=\int_{\dd X}|\vp(\xi)|^2\mathcal{I}_s[{\bf1}](\xi)d\mu_o(\xi)-(\mathcal{I}_s[\vp],\vp)
\end{align*}
for any $\vp\in L^2(\dd X,\mu_o)$.
If $\mathcal{M}_{s}$ stands for the operator over $L^2(\dd X,\mu_o)$ defined as $\mathcal{M}_{s}[\vp](\xi)=\sqrt{\mathcal{I}_s[{\bf1}]}(\xi).\vp(\xi)$ one has:
$$\mathcal{M}_{s}=\mathcal{I}_s+\mathfrak{D}_s$$
as positive operators over $L^2(\dd X,\mu_o)$.
In particular Lemma \ref{lem:is} implies:
$$(\mathcal{I}_s[\vp],\vp)\le(\mathcal{M}_{s}\vp,\vp)\prec_s \|\vp\|_2^2$$
 for any $\vp\in \cC(\dd X)$ and thus $L^2(\dd X,\mu_o)\subset \mathcal{H}_s$ since continuous functions over $\dd X$ are dense in $L^2(\dd X,\mu_o)$.

On the other hand $\mathcal{I}_s$ is positive over $L^2(\dd X,\mu_o)$ and therefore the operator $R_{s,\e}=(\mathcal{I}_s+\e{\bf I})^{-1}$ is well defined and positive over $L^2(\dd X,\mu_o)$ for any $\e>0$.
Observe that:
\begin{align*}
&\sqrt{R_{s,\e}}\circ\mathcal{I}_s\circ \sqrt{R_{s,\e}}={\bf I}-\e R_{s,\e}
\end{align*}
and therefore:
\begin{align*}
\sqrt{R_{s,\e}}\circ \mathcal{M}_s\circ \sqrt{R_{s,\e}}&=\sqrt{R_{s,\e}}\circ[\mathcal{I}_s+\mathfrak{D}_s]\circ \sqrt{R_{s,\e}}\\
&={\bf I}-\e R_{s,\e}+\mathfrak{D}_{s,\e}'\\
\end{align*}
where $\mathfrak{D}_{s,\e}'=\sqrt{R_{s,\e}}\circ \mathfrak{D}_{s}\circ \sqrt{R_{s,\e}}\ge0$.
It follows that:
$${\bf I}\le \sqrt{R_{s,\e}}\circ \mathcal{M}_s\circ \sqrt{R_{s,\e}}+\e R_{s,\e}$$
as operators over $L^2(\dd X,\mu_o)$ for any $\e>0$.

Together with the relation:
$$\mathcal{I}_s\circ R_{s,\e}\circ \mathcal{I}_s=\mathcal{I}_s-\e{\bf I}+\e^2 R_{s,\e}$$
one obtain:
\begin{align*}
\|\vp\|^2_2&\le ( \mathcal{M}_s\circ \sqrt{R_{s,\e}}[ \vp],\sqrt{R_{s,\e}}[ \vp])+\e(R_{s,\e}[\vp],\vp)\\
&\asymp_s(1+\e)(R_{s,\e}[ \vp],\vp)\\
&=(1+\e)\|\vp\|_{1-s}^2-(1+\e)\e\|\psi\|_2^2+(1+\e)\e^2(R_{s,\e}[ \psi],\psi)
\end{align*}
where $\vp=\mathcal{I}_s[\psi]\in\mathcal{F}_{1-s}'\simeq\mathcal{I}_s[\cC(\dd X)]$ and any $\e>0$.
To conclude observe that:
$$\| R_{s,\e}\|_{L^2\rightarrow L^2}\le\sup_{t\ge\e}\frac{1}{t}=\frac{1}{\e}$$
for any $\e>0$ which implies that $\e^2R_{s,\e}\xrightarrow{\e\rightarrow 0}0$ in operator norm
and therefore
$$\|\vp\|^2_2\prec_s\|\vp\|_{1-s}^2$$
In other words $\mathcal{H}_{1-s}\subset L^2(\dd X,\mu_o)$.
\end{proof}

Given $0\le\gs\le1$, $t\ge0$ and $f_1,f_2\in\cC(\ol{X})$ we define the operator $\mathcal{A}_{\gs,t}[f_1,f_2]$ over $L^2(\dd X,\mu_o)$ as:
$$\mathcal{A}_{\gs,t}[f_1,f_2]:=\int_{G} \frac{\pi_{\gs}(g)}{\gTh_{\gs}[{\bf1}](g.o)}f_1(g.o)f_2(g^{-1}.o)d\nu_{o,t}(g)$$

\begin{lem}\label{lem:unifbound}
For any $\half< s\le1$ and $f_1,f_2\in\cC(\ol{X})$ the sequence of operators $(\mathcal{A}_{1-s,t}[f_1,f_2])_t$ over $L^2(\dd X,\mu_o)$ is uniformly bounded with
$$\|\mathcal{A}_{s,t}[f_1,f_2]\|\prec_s\|f_1\|_\infty\|f_2\|_\infty$$
\end{lem}
\begin{proof}
Let us denote $\mathcal{P}_{1-s,t}:=\mathcal{A}_{1-s,t}[{\bf1},{\bf 1}]$.

First observe that:
\begin{align*}
&\mathcal{P}_{1-s,t}[{\bf1}](\xi)=\int_{G}\frac{\pi_{1-s}(g)}{\gTh_{1-s}[{\bf1}](g.o)}[{\bf1}](\xi)d\nu_{o,t}(g)\\
&=\int_{G} \frac{\pi_{1-s}(g)}{\gTh_{1-s}[{\bf1}](g.o)}[{\bf1}](\xi)d\nu_{o,t}(g)\\
&\prec_s \int_{G}e^{2(1-s)\gd\langle\xi,g.o\rangle_o}d\nu_{o,t}(g)
\end{align*}
where the last inequality follows from Lemma \ref{lem:poies}.

On the other hand Lemma \ref{lem:count} and \ref{lem:vitali} imply:
\begin{align*}
&\mathcal{P}_{1-s,t}[{\bf1}](\xi)\prec_{s} 1+2(1-s)\gd\int_0^{t.R}e^{2(1-s)\gd u}e^{-\gd tR}|\{g\in S^*_{o,G}\,|\,\langle \xi,g.o\rangle_o\ge u\}|du\\
&\prec 1+2(1-s)\gd\int_0^{t.R}e^{2(1-s)\gd u}e^{-\gd tR}|{\bf C}^+_{o}(\xi;u,r)\cap S_{o,G}(t)|du\\
&\asymp_s  1+2(1-s)\gd\int_0^{t.R}e^{(1-2s)\gd u}du=\frac{1}{(2s-1)}+[1-\frac{1}{(2s-1)}]e^{(1-2s)\gd t.R}
\end{align*}

Similarly one has:
\begin{align*}
&\mathcal{P}^*_{1-s,t}[{\bf1}](\xi)=\int_{G}\frac{\pi_{s}(g^{-1})}{\gTh_{s}[{\bf1}](g^{-1}.o)}[{\bf1}](\xi)d\nu_{o,t}(g)\\
&\prec_s  e^{(1-2s)tR}\int_{G}e^{2s\gd\langle\xi,g^{-1}.o\rangle_o}d\nu_{o,t}(g)\\
&\le e^{(1-2s)\gd tR}+2s\gd e^{(1-2s)\gd tR}\int_0^{t.R}e^{2s\gd u}e^{-\gd tR}|{\bf C}^-_{o}(\xi;u,r)\cap S_{o,G}(t))|du\\
&\asymp e^{(1-2s)\gd tR}+2s\gd e^{(1-2s)\gd tR}\int_0^{t.R}e^{(2s-1)\gd u}du=1+\frac{1}{(2s-1)}[1-e^{(1-2s)\gd tR}]
\end{align*}
In other words $\mathcal{P}_{1-s,t}[{\bf1}]$ and $\mathcal{P}^*_{1-s,t}[{\bf1}]$ are uniformly bounded.

Using the Cauchy-Schwarz inequality with respect to the measure
$$\frac{1}{\gTh_{1-s}(g.o)}[\frac{d\mu_{g.o}}{d\mu_o}]^{1-s}(\xi)d\mu_o(\xi)d\nu_{o,t}(g)$$
we eventually obtain:
\begin{align*}
|(&\mathcal{A}_{1-s,t}[f_1,f_2](\vp),\phi)|^2\\
&=|\int_{G\times\dd X}\vp(g^{-1}\xi)\ol{\phi}(\xi)f_1(g.o)f_2(g^{-1}.o)[\frac{d\mu_{g.o}}{d\mu_o}]^{1-s}(\xi)d\mu_o(\xi)\frac{d\nu_{o,t}(g)}{\gTh_{1-s}(g.o)}|^2\\
&\le \|f_1\|^2_\infty\|f_2\|^2_\infty\int_{G\times\dd X}|\vp(g^{-1}\xi)|^2[\frac{d\mu_{g.o}}{d\mu_o}]^{1-s}(\xi)d\mu_o(\xi)\frac{d\nu_{o,t}(g)}{\gTh_{1-s}(g.o)}\\
&\quad\times\int_{G\times\dd X}|\phi(\xi)|^2[\frac{d\mu_{g.o}}{d\mu_o}]^{1-s}(\xi)d\mu_o(\xi)\frac{d\nu_{o,t}(g)}{\gTh_{1-s}(g.o)}\\
&=  \|f_1\|^2_\infty\|f_2\|^2_\infty(\mathcal{P}_{1-s,t}[|\vp|^2],{\bf1}).(\mathcal{P}_{1-s,t}[{\bf1}],|\phi|^2)\\
&\le \|f_1\|^2_\infty\|f_2\|^2_\infty\|\mathcal{P}_{1-s,t}[{\bf1}]\|_\infty\|\mathcal{P}^*_{1-s,t}[{\bf1}]\|_\infty\|\vp\|_2^2\|\phi\|_2^2
\end{align*} 
 for any $\vp,\phi\in L^2(\dd X,\mu_o)$ or equivalently:
 $$\|\mathcal{A}_{1-s,t}[f_1,f_2]\|\le\sqrt{\|\mathcal{P}_{1-s,t}[{\bf1}]\|_\infty\|\mathcal{P}^*_{1-s,t}[{\bf1}]\|_\infty}\|f_1\|_\infty\|f_2\|_\infty$$
\end{proof}

As a consequence of Proposition \ref{prop:incl} and Lemma \ref{lem:unifbound} one has:
$$\|\mathcal{A}_{1-s,t}[f_1,f_2](v)\|_s\le\|\mathcal{A}_{1-s,t}[f_1,f_2](v)\|_2\prec_{s,f_1,f_2}\|v\|_2\le\|v\|_{1-s}$$
for any $v\in \mathcal{H}_{1-s}$ and $t\ge0$.

Therefore the sequence of operators $\mathcal{A}_{1-s,t}[f_1,f_2]|_{\mathcal{H}_{1-s}}^{\mathcal{H}_{s}}$ with $\half< s\le1$ and $f_1,f_2\in\cC(\ol{X})$ is uniformly bounded with:
$$\|\mathcal{A}_{s,t}[f_1,f_2]\|_{\mathcal{H}_{1-s}\rightarrow \mathcal{H}_{s}}\prec_s\|f_1\|_\infty\|f_2\|_\infty$$

If $[f_i|_{\dd X}]_s$, $i=1,2$, denote the class of $f_i$ in $\mathcal{H}_s$.
It follows from Proposition \ref{prop:matequi} together with Lemma \ref{lem:unifbound} :
\begin{cor}\label{cor:bilan} 
The averages of matrix coefficients:
$$\int_{G} \frac{Q(\pi_{s}(g)v,w)}{\gTh_{1-s}[{\bf1}](g.o)}.f_1(g.o)f_2(g^{-1}.o)d\nu_{o,t}(g)$$
equidistribute when $t$ goes to infinity toward:
$$Q(\mathcal{I}_s[v],[f_2|_{\dd X}]_s)Q(w,[f_1|_{\dd X}]_{s})$$
for all $v\in\mathcal{H}_s$ and $w\in\mathcal{H}_{1-s}$ and similarly:
$$\int_{G} \frac{Q(\mathcal{I}_s[\pi_s(g)v],w)}{\gTh_{s}[{\bf1}](g^{-1}.o)}.f_1(g.o)f_2(g^{-1}.o)d\nu_{o,t}(g)\rightarrow Q(\mathcal{I}_s[v],[f_2|_{\dd X}]_s)Q(\mathcal{I}_s[w],[f_1|_{\dd X}]_s)$$
for all $v,w\in \mathcal{H}_s$.
\end{cor}

\section{Properties of complementary series}\label{sec:end}$ $\\
In this section we  investigate the properties of the family of unitary representations $(\mathcal{H}_s)_{s\in[0,1]}$ constructed in Section \ref{sec:pos}.
The proof of Theorem 2 will follow from the propositions \ref{prop:irr}, \ref{prop:weak}, \ref{prop:fell} and Corollary \ref{cor:noeq}.\\

A straightforward application of Corollary \ref{cor:bilan} is the following extension of the Martin-Poisson correspondence:
\begin{cor}
For $\half<s\le1$ the operators:
$$\gTh_{1-s}:\cC(\dd X)\simeq\mathcal{F}_{1-s}\rightarrow L^\infty(X,m)$$ and 
$$\Xi_s:\cC(\dd X)\simeq\mathcal{F}_{s}\rightarrow L^\infty(X,m);\quad x\mapsto \gTh_{1-s}[\mathcal{I}_s[\vp]](x)$$ 
are injective intertwiners.
\end{cor}
Note that $\text{Ran}[\Xi_s]\subset \text{Ran}[\gTh_{1-s}]$, $\Xi_s[\vp]\asymp_s\gTh_s[\vp]$ and
$$\mathcal{I}_s|_{\mathcal{F}_s}=[\gTh_{1-s}|^{\text{Ran}[\gTh_{1-s}]}]^{-1}\circ\Xi_s|_{\mathcal{F}_s}$$
Compare with the formula introduced in \cite{MR710827} Chapter 4.

\begin{prop}\label{prop:irr}
The unitary complementary representations $(\pi_s)_s$ are irreducible for all $0\le s\le1$.
\end{prop}
\begin{proof}
The case $s=\half$ is proved in \cite{Garncarek:2014aa} and the cases $s=0,1$ are trivial. 
Since $\pi_s^*\simeq\pi_{1-s}$ let us assume $\half<s<1$.

As proved in Section \ref{sec:anamat} the sequence of operators $(\mathcal{P}_{s,t})_t$ defined as $\mathcal{P}_{s,t}=\mathcal{A}_{s,t}[{\bf1},{\bf1}]$
,for $t\ge0$, converges for the weak operator topology to the projection, $\mathcal{P}_s$, over the constant vector ${\bf1}\in\mathcal{H}_s$ in $\mathcal{B}[\mathcal{H}_s]$.

Since operators $(\mathcal{P}_{s,t})_t$ belongs to the von Neumann algebra generated by $\pi_s$ inside of $\mathcal{B}[\mathcal{H}_s]$, in order to prove the irreducibility of $\pi_s$ it is enough to prove that ${\bf1}$ is cyclic for $\pi_s$ \cite{MR0352996}.

Let $v\in \ol{\text{Span}[\pi_s(g){\bf1}]}^\bot\subset\mathcal{H}_{s}$, then for any $f\in\cC(\dd X)\simeq \mathcal{F}_s$ one has:
$$Q_s(\mathcal{A}_{s,t}[ f,{\bf1}]{\bf1},v)=\int_G\frac{Q(\mathcal{I}_s[\pi_s(g){\bf1}],v)}{\gTh_s[{\bf1}](g^{-1}.o)}.\hat{f}(g.o)d\nu_{o,t}(g)=0.$$
On the other hand Corollary \ref{cor:bilan}  implies:
$$\int_{G} \frac{Q(\mathcal{I}_s[\pi_s(g){\bf1}],v)}{\gTh_s[{\bf1}](g^{-1}.o)}.\hat{f}(g.o)d\nu_{o,t}(g)\rightarrow Q_s({\bf1})Q(\mathcal{I}_s(f),v)$$
, in other words $Q(\mathcal{I}_s(f),v)=0$ for any $f\in \cC(\dd X)$. 
Since $\mathcal{F}_s\simeq \cC(\dd X)$ is dense in $\mathcal{H}_s$, $v=0$ and the vector ${\bf1}$ is cyclic.
\end{proof}

Given $\vp,\phi\in\cC(\dd X)\simeq\mathcal{F}_s$ Lemma \ref{lem:poies} implies the following estimate on the rate of decay holds:
$$|Q_s(\pi_s(g)\vp,\phi)|\prec_s\|\vp\|_\infty\|\phi\|_\infty\gTh_s[{\bf1}](g.o)\asymp_s\|\vp\|_\infty\|\phi\|_\infty e^{-[\half-|s-\half|]\gd d(g.o,o)}$$
As proved below this is sharp.

\begin{prop}[Characteristic decay]\label{prop:characterization}
Let $\pi_s$ be the complementary representation of parameter $\half<s\le1$ and $v,w\in\mathcal{H}_s$.
If one can find $\e>0$ such that:
$$Q_s(\pi_s(g)v,w)=O(e^{-[\half-|s-\half|]\gd d(g.o,o)-\e d(g.o,o)})$$
then $Q_s(\pi_s(\bullet)v,w)=0$.
\end{prop}

\begin{proof}
Lemma \ref{lem:poies} implies: 
\begin{align*}
&Q_s(\mathcal{A}_{s,t}[f_1,f_2]v,w)=\int_{G}\frac{Q_s(\pi_s(g)v,w)}{\gTh_s[{\bf1}](g.o)}f_1(g.o)f_2(g^{-1}.o)d\nu_{o,t}(g)\\
&\prec_s \|f_1\|_\infty\|f_2\|_\infty\int_{G}e^{-\e d(g.o,o)}d\nu_{o,t}(g)\asymp\|f_1\|_\infty\|f_2\|_\infty e^{-\e tR}\rightarrow 0\\
\end{align*}
for any continuous functions $f_1$ and $f_2$ in $\cC(\ol{X})$

On the other hand:
$$\lim_tQ_s(\mathcal{A}_{s,t}[f_1,f_2]v,w)=Q_s(v,[f_1|_{\dd X}]_s)Q_s([f_2|_{\dd X}]_s,w)$$
for all $f_1$ and $f_2$ in $\cC(\ol{X})$.

It follows that $v$ or $w$ must be zero and therefore $Q_s(\pi_s(g)v,w)=0$ for all $g\in G$.
\end{proof}

\begin{cor}\label{cor:noeq}
Let $\pi_s$ and $\pi_{s'}$ be two complementary representations of parameter $0\le s,s'\le1$.
Then $\pi_s$ and $\pi_{s'}$ are unitary equivalent if and only if $s'=1-s$
\end{cor}
\begin{proof}
Since $\pi_s\simeq\pi_{1-s}$ for any $0\le s\le 1$ one can assume $\half< s<s'\le1$.
Suppose one can find $T_{s\rightarrow s'}:\mathcal{H}_s\rightarrow\mathcal{H}_{s'}$ unitary intertwiner between $\pi_s$ and $\pi_{s'}$. 

Then $T_{s\rightarrow {s'}}{\bf1}\neq0$ and:
$$Q_{s'}(\pi_{s'}(g)T_{s\rightarrow {s'}}{\bf1},T_{s\rightarrow {s'}}{\bf1})= Q_s(\pi_s(g){\bf1},{\bf1})$$
It follows that:
$$Q_{s'}(\pi_{s'}(g)T_{s\rightarrow {s'}}{\bf1},T_{s\rightarrow {s'}}{\bf1})=O(e^{-[\half-|s'-\half|]\gd d(g.o,o)-|s'-s| \gd d(g.o,o)})$$
and Proposition \ref{prop:characterization} implies $T_{s\rightarrow {s'}}{\bf1}=0$ which is a contradiction.
\end{proof}

\begin{prop}\label{prop:weak}
The unitary representations $(\mathcal{H}_s)_s$ are not weakly contained in the regular representation of $G$ except for $s=\half$.
\end{prop}
\begin{proof}
The weak containment of $\mathcal{H}_\half$ inside of the regular representation is well known and follows from \cite{MR1293309} with \cite{MR1209424}.
Assume $\mathcal{H}_s$ is weakly contained inside of the regular representation $\gl$ of $G$ for $s>\half$.
The spectral transfer principal implies:
$$\|\pi_s(f)\|\le\|\gl(f)\|$$
for all $f\in L^1(G,dg)$
and thus 
$$\int_G Q_s(\pi_s(g){\bf1},{\bf1})d\nu_{o,t}(g)\prec_s\|\pi_s(\nu_{o,t})\|\le\|\gl(\nu_{o,t})\|$$
On one hand the Corollary \ref{cor:bilan} implies 
$$\int_G Q_s(\pi_s(g){\bf1},{\bf1})d\nu_{o,t}(g)\asymp e^{-(1-s)\gd t}$$
On the other hand, using the hyperbolicity of $G$, the Haagerup inequality \cite{MR943303} implies:
$$\|\gl(\nu_{o,t})\|\le Q(t)\|\nu_{o,t}\|_2\asymp Q(t)e^{-\half\gd t}$$
where $Q$ is polynomial function which is a contradiction.
\end{proof}

\begin{prop}\label{prop:fell}
The family of representations $(\mathcal{H}_s)_{s\in[0,1]}$ is continuously parametrized for the Fell topology over the unitary dual of $G$.
\end{prop}
\begin{proof}
Since $\pi^*_s=\pi_{1-s}$ for all $0\le s\le 1$ it is enough to prove the continuity over the interval $[\half,1]$.
Moreover the continuity at $\mathcal{H}_{s_0}$ for $s_0>\half$ follows from the very definition of those representations together with the dominated convergence theorem.
One can therefore assume $s_0=\half$.

The argument consists to prove that any $\mathcal{H}_{\half}$-matrix coefficient is a limit when $s$ goes to $\half$ of matrix coefficients in $\mathcal{H}_{s}$.
First observe that for any $g\in G$ and $v\in L^2(\dd X,\mu_o)$ one has:
 \begin{align*}
\|\pi_s(g)v\|_2^2&=\int_{\dd X}e^{2s\gd b_\xi(g.o,o)}|v(g^{-1}\xi)|^2d\mu_o(\xi)\\
&=\int_{\dd X}e^{(1-2s)\gd b_\xi(g^{-1}.o,o)}|v(\xi)|^2d\mu_o(\xi)\\
&\le e^{(2s-1)\gd d(g.o,o)}\|v\|_2^2
\end{align*}
in other words $\|\pi_s(g)\|_{L^2\rightarrow L^2}\le e^{\gd d(g.o,o)}$.
Using Proposition \ref{prop:incl} proof the following relation holds over $\mathcal{B}(L^2(\dd X,\mu_o))$:
$$\mathcal{M}_s-\mathcal{I}_s=\mathfrak{D}_s$$
and Lemma \ref{lem:is} implies that the operators $\sqrt{(2s-1)}\mathcal{M}_s\in L^\infty(\dd X,\mu_o)\subset \mathcal{B}(L^2(\dd X,\mu_o))$ with $\half<s\le1$ satisfy:
$$\sqrt{(2s-1)}\mathcal{M}_s\asymp 1$$
uniformly over $s$.
Extracting a subsequence if necessary one can assume that $\sqrt{(2s-1)}\mathcal{M}_s$ converges for the weak operator topology to $\mathcal{M}^0_\half$ when $s>\half$ goes to $\half$ .
Note that the uniform lower bound given by Lemma \ref{lem:is} implies that $\mathcal{M}^0_\half$ is positive and invertible.

Given a Lipschitz function $\vp\in\text{Lips}(\dd X)$ one has:
\begin{align*}
(\mathfrak{D}_s\vp,\vp)&=\half\int_{\dd^2X}\frac{|\vp(\xi)-\vp(\eta)|^2}{d^{2(1-s)\gd}_{o}(\xi,\eta)}d\mu_o(\xi)d\mu_o(\eta)\\
&\prec_\vp\half\int_{\dd^2X}d^{2\gd[\gd^{-1}+s-1]}_{o}(\xi,\eta)\mu_o(\xi)d\mu_o(\eta)
\end{align*}
that is uniformly bounded for $s\ge\half-\frac{1}{2\gd}$.\\
The Cauchy-schwarz inequality relative to the positive operator $\mathfrak{D}_s$ over $L^2(\dd X,\mu_o)$ implies:
\begin{align*}
(2s-1)|(\mathfrak{D}_s[\pi_s(g)\vp],\vp)|^2&\le (2s-1)(\mathfrak{D}_s[\pi_s(g)\vp],\pi_s(g)\vp)(\mathfrak{D}_s[\vp],\vp)\\
&= (2s-1)[(\mathcal{M}_s[\pi_s(g)\vp],\pi_s(g)\vp)-(\mathcal{I}_s[\vp],\vp)](\mathfrak{D}_s[\vp],\vp)\\
&\prec \sqrt{(2s-1)}[\|\pi_s(g)\vp\|_2^2+\|\vp\|_2^2](\mathfrak{D}_s[\vp],\vp)\\
&+(2s-1)|(\mathfrak{D}_s[\vp],\vp)|^2 \rightarrow 0
\end{align*}
On the other hand the dominated convergence theorem implies:
$$(\pi_s(g)\vp,\sqrt{(2s-1)}\mathcal{M}_s[\vp])\rightarrow (\pi_\half(g)\vp,\mathcal{M}^0_\half[\vp])$$
Eventually one has:
\begin{align*}
&\sqrt{(2s-1)}(\mathcal{I}_s\pi_s(g)\vp,\vp)=\sqrt{(2s-1)}([\mathcal{M}_s-\mathfrak{D}_s]\pi_s(g)\vp,\vp)\\
&=(\pi_s(g)\vp,\sqrt{(2s-1)}\mathcal{M}_s[\vp])-\sqrt{(2s-1)}(\mathfrak{D}_s[\pi_s(g)\vp],\vp))\rightarrow (\pi_\half(g)\vp,\mathcal{M}^0_\half[\vp])
\end{align*}
Since $\text{Lips}(\dd X)$ is dense in $L^2(\dd X,\mu_o)$ and operator $\mathcal{M}_\half^0\in L^\infty(\dd X,\mu_o)$ is invertible the claim is proved.
\end{proof}


\nocite{*}
\bibliographystyle{plain}
\bibliography{complementarxiv}
\end{document}